\documentclass[11pt,reqno]{amsart}
\usepackage{mathtools}\mathtoolsset{showonlyrefs} % for \eqref
\usepackage[margin=1in]{geometry}
\usepackage{amsmath,amssymb,amsthm,microtype,graphicx,lmodern}
\usepackage[authoryear]{natbib}
\usepackage[hidelinks]{hyperref}
\allowdisplaybreaks%
\numberwithin{equation}{section}%
\numberwithin{figure}{section}%
\numberwithin{table}{section}
\theoremstyle{plain}
\newtheorem{theorem}{Theorem}[section]

\newtheorem{lemma}[theorem]{Lemma}
\newtheorem{corollary}[theorem]{Corollary}
\theoremstyle{definition}

\newtheorem{remark}[theorem]{Remark}
\newtheorem{problem}[theorem]{Problem}

\title[An optimal Poincaré inequality for the complex Ginibre log-gas]{An optimal Poincaré inequality\\for the complex Ginibre log-gas}
\author{Djalil Chafaï}
\address{CEREMADE Université Paris-Dauphine -- PSL, and DMA, École normale
  supérieure -- PSL, France}
\urladdr{https://djalil.chafai.net/}
\email{djalil@chafai.net}
\date{Summer 2026}
\begin{document}
\begin{abstract}
  We establish an optimal Poincaré inequality for real-valued symmetric
  observables of the complex Ginibre log-gas. Equality is attained by the real
  and imaginary parts of the center-of-mass observable. Equivalently, we
  determine the exact spectral gap of the associated overdamped Langevin
  dynamics, for real symmetric observables. The Hessian of the energy of this
  log-gas is unbounded below, so standard convexity arguments do not directly
  apply. The proof instead combines a Vandermonde transform, a holomorphic
  projection, and a complex Gaussian d-bar spectral-gap estimate,
  corresponding to the constant-curvature case of the Hörmander--Berndtsson
  estimate. It is short and self-contained. It remains valid, beyond the
  quadratic potential, for rotationally invariant additive plurisubharmonic
  potential perturbations. Additionally, we provide four alternative proofs,
  two of which yield sum-of-squares formulas for the deficit, based
  respectively on a Hermite expansion and on an integrated Bochner--Kodaira
  formula, the other two use the spectral analysis of a two-sided
  number-operator factorization and a Hermite--Slater polynomial expansion. We
  furthermore present new results related to linear statistics, Gaussian
  factorization, polynomial eigenfunctions, log-Sobolev inequalities, matrix
  lift and eigenvector overlaps, and non-quadratic potentials.
\end{abstract}
\maketitle
{\footnotesize\tableofcontents}
\section{Introduction and main result}

This article provides a sharp solution to a natural problem about a
probability measure on $\mathbb{C}^n\equiv\mathbb{R}^{2n}$, $n\geq1$, known as
the complex Ginibre ensemble log-gas, or Coulomb gas. It is given by
\begin{equation}\label{eq:mu}
  \mathrm{d}\mu_n
  =\frac{1}{c_n}\mathrm{e}^{-n\sum_{k=1}^n|z_k|^2}\prod_{j<k}|z_j-z_k|^2\mathrm{d}z
  =\frac{1}{c_n'}|V_n(z)|^2\mathrm{d}\gamma_n  
\end{equation}
where $\mathrm{d}z$ is the Lebesgue measure, $\gamma_n$ is the product complex
Gaussian distribution $\mathcal{N}_{\mathbb{C}^n}(0,1/n)$ on $\mathbb{C}^n$
with density $z\mapsto(\frac{n}{\pi})^n\mathrm{e}^{-n|z|^2}$,
$|z|^2=\sum_{k=1}^n|z_k|^2$, and $V_n(z)$ is the Vandermonde determinant
\begin{equation}\label{eq:Vn}
  V_n(z)=V_n(z_1,\ldots,z_n)=\prod_{1\leq j<k\leq n}(z_k-z_j)=\det{[z_k^{j-1}]}_{1\leq j,k\leq n}.
\end{equation}
The quantities $c_n$ and $c_n'$ are explicitly computable normalizing
constants\footnote{Can be seen as particular complex Selberg integrals,
  $c_n=\pi^nn^{-n(n+1)/2}\prod_{k=1}^n k!$, and $c'_n=(n/\pi)^nc_n$.}. We can
also view the probability measure $\mu_n$ as a Boltzmann--Gibbs measure at
inverse temperature $\beta_n=n^2$, namely
\begin{equation}\label{eq:boltzmanngibbs}
  \mathrm{d}\mu_n=\frac{1}{c_n}\mathrm{e}^{-\beta_n E_n(z)}\mathrm{d}z,
\end{equation}
with energy 
\begin{equation}\label{eq:Hn}
  E_n(z)=E_n(z_1,\ldots,z_n)=\frac{1}{n}\sum_{k=1}^n|z_k|^2+\frac{1}{n^2}\sum_{j\neq k}\log\frac{1}{|z_j-z_k|}.
\end{equation}
Following Eugene Wigner and Freeman Dyson, this allows us to interpret $\mu_n$
as a planar Coulomb gas of $n$ interacting particles, with equal charges,
confined by a quadratic potential. The two-body interaction is a Coulomb
repulsion since it is logarithmic in dimension two. This log-gas appears in
many contexts in probability theory and mathematical physics, notably in
random matrix theory for the law of the spectrum of complex Gaussian random
matrices, see
\cite{Ginibre1965,ForresterLogGases,ChafaiAspects,ByunForresterGinibre}, in
the wave functions of the fractional quantum Hall effect, see \cite{Laughlin},
and in the distribution of vortices related to the Ginzburg--Landau equations
of superconductivity, see \cite{SerfatyCoulombGases}. It also describes the
joint law of the positions in the ground state of non-interacting fermions in
a rapidly rotating two-dimensional harmonic trap, when the fermions occupy the
lowest Landau level, see \cite{LacroixMajumdarSchehrGinibreFermions}. We also
refer to \cite{LewinCoulombRieszGases,SerfatyLecturesCoulombRiesz} for the
analysis and statistical mechanics of such log-gases.

The non-zero set of the density of $\mu_n$ is
$\mathbb{C}^n_{\neq}=\mathbb{C}^n\setminus\mathbb{C}^n_{=}$ with
$\mathbb{C}^n_{=}=\cup_{j<k}\{z\in\mathbb{C}^n:z_j=z_k\}$.

\subsection{Main result}

We say that a function $f$ defined on $\mathbb{C}^n$ is \emph{symmetric} when
\begin{equation}\label{eq:symmetric}
  f(z_{\sigma(1)},\ldots,z_{\sigma(n)})=f(z_1,\ldots,z_n)
\end{equation}
for all $z\in\mathbb{C}^n$ and every permutation $\sigma\in\mathfrak{S}_n$
where $\mathfrak{S}_n$ is the symmetric group. We denote by $L^2(\mu_n)$ and
$L^2(\gamma_n)$ the Hilbert spaces of square integrable classes of functions
from $\mathbb{C}^n\to\mathbb{C}$, with respect to $\mu_n$ and $\gamma_n$
respectively. We denote by $\mathcal{C}^\infty_c(\mathbb{C}^n,\mathbb{R})$ the
space of real-valued and compactly supported functions of class
$\mathcal{C}^\infty$ defined on $\mathbb{C}^n$. For all
$f\in\mathcal{C}^\infty_c(\mathbb{C}^n,\mathbb{R})$, we denote by $\nabla f$
the Euclidean gradient of $f$ on $\mathbb{R}^{2n}\equiv\mathbb{C}^n$ and by
$|\nabla f|$ its Euclidean norm. We denote by $H^1_{\mathbb{R}}(\mu_n)$ the
Sobolev space obtained by taking the completion of
$\mathcal{C}^\infty_c(\mathbb{C}^n,\mathbb{R})$ for the norm defined by
\begin{equation}\label{eq:graphnorm}
  \|f\|_{H^1_{\mathbb{R}}(\mu_n)}=\sqrt{\|f\|_{L^2(\mu_n)}^2+\|\nabla f\|_{L^2(\mu_n)}^2}.
\end{equation}

% It turns out that $H^1_{\mathbb{R}}(\mu_n)$ is also the closure of
% $\mathcal{C}^\infty_c(\mathbb{C}^n_{\neq},\mathbb{R})$ for this norm where
% $\mathbb{C}^n_{\neq}=\mathbb{C}^n\setminus\mathbb{C}^n_{=}$ with
% $\mathbb{C}^n_{=}=\cup_{j<k}\{z\in\mathbb{C}^n:z_j=z_k\}$, since
% $\mathbb{C}^n_{=}$ has zero weighted capacity.

This definition of $H^1_{\mathbb{R}}(\mu_n)$ is useful for approximation. See
also Lemma \ref{le:gradient:closure:domain} for another viewpoint.

% Actually $H^1_{\mathbb{R}}(\mu_n)$ is also the space of real-valued
% $f\in L^2(\mu_n)$ such that $f\in H^1_{\mathrm{loc}}(\mathbb{C}^n_{\neq})$
% and $|\nabla f|\in L^2(\mu_n)$, but this is not used here.

Let us see $z_1,\ldots,z_n$ as the locations of the $n$ particles of a planar
gas or cloud of particles. Two very natural real observables of the
center-of-mass $\frac{1}{n}(z_1+\cdots+z_n)$ play a special role in this
article. They are the symmetric real-valued linear functions defined for all
$z\in\mathbb{C}^n$ by
\begin{equation}
  \varphi_\Re(z)=\Re\sum_{k=1}^nz_k=\sum_{k=1}^n\Re z_k
  \quad\text{and}\quad
  \varphi_\Im(z)=\Im\sum_{k=1}^n z_k=\sum_{k=1}^n\Im z_k.
\end{equation}
They are $\mathcal{C}^\infty(\mathbb{C}^n,\mathbb{R})$, and standard
approximation arguments show that they belong to $H^1_{\mathbb{R}}(\mu_n)$.

\begin{theorem}[Optimal Poincaré inequality]\label{th:main}
  For all $n\geq1$, for all symmetric $f\in H^1_{\mathbb{R}}(\mu_n)$, 
  \begin{equation}\label{eq:pi}
    \operatorname{Var}_{\mu_n}(f) \leq \frac{1}{2}\mathcal{E}_{\mu_n}(f,f)
  \end{equation}
  where
  \begin{equation}
    \operatorname{Var}_{\mu_n}(f)
    =\operatorname{Cov}_{\mu_n}(f,f)=\int_{\mathbb{C}^n} f^2\mathrm{d}\mu_n-\Bigl(\int_{\mathbb{C}^n}f\mathrm{d}\mu_n\Bigr)^2
    \quad\text{and}\quad
    \mathcal{E}_{\mu_n}(f,f)=\frac{1}{n}\int_{\mathbb{C}^n}|\nabla f|^2\mathrm{d}\mu_n.
  \end{equation}
  Moreover, equality is achieved for $f=a\varphi_\Re+b\varphi_\Im$, for all
  $a,b\in\mathbb{R}$.
\end{theorem}

The proof of Theorem \ref{th:main}, given in Section \ref{se:proof:th:main}, is
based on a Vandermonde or ground-state transform, a holomorphic projection, a
Gaussian complex-analytic inequality which is a special case of the
$\bar\partial$ Hörmander--Berndtsson estimate, a divisibility property, and a
holomorphic--antiholomorphic projection geometry lemma whose real-valued
consequence is a half-distance estimate. This approach is suitable for the
extension, beyond the quadratic potential case, to rotationally invariant
additive plurisubharmonic perturbations of the quadratic potential, see
Theorem \ref{th:nonquad}.

We also give four alternative proofs of Theorem \ref{th:main}, more Gaussian,
two using the integrated Bakry--Émery curvature, see Section \ref{se:BE}, one
using a complex operator and a spectral approach in Section
\ref{se:proof:th:main:spectral}, and one using a Hermite--Slater polynomial
expansion in Section \ref{se:proof:th:main:polynomials}.

The mechanism behind the proof of Theorem~\ref{th:main} is stronger than the
usual correlation-kernel identities associated with the determinantal
structure of the log-gas that appear in random matrix theory, in the sense
that the result goes strictly beyond additive observables also known as linear
statistics. A key point is that the multiplication of the Vandermonde
determinant by a symmetric function preserves its antisymmetry. In quantum
mechanics, this is precisely the fermionic antisymmetry encoded by Slater
determinants built from Bargmann--Fock orbitals.

For $n=1$, $\mu_1=\gamma_1$ is strongly log-concave. Let us consider now the
case $n\geq2$. The probability measure $\mu_n$ is not log-concave, in other
words the energy $E_n$ is not convex. It is built from the planar Coulomb
kernel $-\log\left|\cdot\right|$ which is not even subharmonic. Actually it is
superharmonic, indeed, let us recall the fundamental equation
$\Delta\log\left|\cdot\right|=2\pi\delta_0$ in the sense of distributions. In
the same vein, the infimum of the eigenvalues of the Hessian matrix of $E_n$
is $-\infty$, and thus the Bakry--Émery curvature is $-\infty$, making Theorem
\ref{th:main} particularly interesting. Still the Poincaré inequality remains
equivalent to the integrated $\Gamma_2$ criterion, see
\cite{BakryGentilLedoux}, or \cite{AneEtAlLSI}.

The Poincaré inequality is a classical concept in probabilistic functional
analysis, which goes back to \cite{zbMATH02695433,Poincare1890}, and which has
plenty of applications, see \citet[Ch.~4]{BakryGentilLedoux} and references
therein. In the present note, we use the Hörmander--Berndtsson estimate for
the standard Gaussian distribution, see \citet[Th.~1.3]{BerndtssonThingsDbar},
\citet[Th.~1.1.1]{BerndtssonDbarMethods}, inspired by
\citet[Th.~2.2.1]{HormanderL2Dbar},
\citet[Ch.~4]{HormanderComplexAnalysisSeveralVariables}. Interestingly, and as
noticed in \citet[p.~402]{CorderoErausquinBerndtssonPrekopa}, a real case
version of this estimate is nothing else but the Poincaré-type inequality for
log-concave measures obtained in \citet[Th.~4.1]{BrascampLiebExtensions}.

Dimension-free functional inequalities for weakly interacting Boltzmann--Gibbs
measures have a long history, see \cite{LedouxUnboundedSpinRevisited} for a
concise account of the perturbative theory for unbounded spin systems, and
\cite*{GuillinLiuWuZhang} for explicit uniform Poincaré and logarithmic Sobolev
inequalities for mean-field particle systems. The singular Coulomb interaction
considered here falls outside these regular perturbative frameworks, since its
Hessian is unbounded from below.

The functional inequalities obtained in this article can be used to get
concentration of measure for Lipschitz functions, which can be compared for
instance to the results of \cite*{ChafaiHardyMaida}, and the quantification of
the cutoff phenomenon, as in \cite*{ChafaiFathiCutoff}.

\subsection{Real case}

Let $\gamma_n^{\mathbb{R}}$ be the product real Gaussian distribution
$\mathcal{N}_{\mathbb{R}^n}(0,1/n)$ on $\mathbb{R}^n$ with density
$x\mapsto(\frac{n}{2\pi})^{\frac{n}{2}}\mathrm{e}^{-\frac{n}{2}|x|^2}$. In
random matrix theory, the analogue of $\mu_n$ on $\mathbb{R}^n$ is known as
the Gaussian Unitary Ensemble (GUE) log-gas, given by
\begin{equation}\label{eq:muR}
  \mathrm{d}\mu_n^{\mathbb{R}}
  =\frac{1}{c_n^{\mathbb{R}}}\mathrm{e}^{-\frac{n}{2}\sum_{k=1}^n|x_k|^2}\prod_{j<k}|x_j-x_k|^2\mathrm{d}x
  =\frac{1}{c_n^{'\mathbb{R}}}\prod_{j<k}|x_j-x_k|^2\mathrm{d}\gamma_n^{\mathbb{R}}.
\end{equation}
The Poincaré inequality for $\mu_n^{\mathbb{R}}$ on symmetric observables is
equivalent to the Poincaré inequality for its normalized restriction
$\widetilde{\mu}_n^{\mathbb{R}}$ on the simplex
$\{x\in\mathbb{R}^n:x_1\leq\cdots\leq x_n\}$, and it turns out that
$\widetilde{\mu}_n^{\mathbb{R}}$ is strongly log-concave, due to the fact that
$u\mapsto-\log(u)$ is convex on $\mathbb{R}_+$. Therefore, the Bakry--Émery
criterion applies, and provides the analogue of Theorem \ref{th:main}, see
\cite{ChafaiLehec}, where an optimal logarithmic Sobolev inequality is also
obtained. We also refer to \cite{RosenzweigSerfatyModulatedLSI} for open
problems in the same spirit beyond this exactly solvable specific model.

\subsection{Nonsymmetric case}

The symmetry hypothesis in Theorem \ref{th:main} cannot be removed. Indeed,
for the nonsymmetric test function $f(z)=\Re z_1$, we get
\begin{equation}
  \mathcal{E}_{\mu_n}(f,f)=\frac{1}{n}\int_{\mathbb{C}^n}|\nabla f|^2\mathrm{d}\mu_n=\frac{1}{n}
  \quad\text{while}\quad
  \int_{\mathbb{C}^n}f\mathrm{d}\mu_n
  =-\int_{\mathbb{C}^n}f\mathrm{d}\mu_n
  =0,
\end{equation}
and, following for instance \cite[Cor.~6.2]{ChafaiAspects},
\begin{equation}
  \int_{\mathbb{C}^n}f^2\mathrm{d}\mu_n
  =\frac{1}{2n}\int_{\mathbb{C}^n}|z|^2\mathrm{d}\mu_n(z)
  =\frac{1}{2n}\frac{n+1}{2}=\frac{n+1}{4n}.
\end{equation}
This shows that the best constant in the Poincaré inequality for $\mu_n$
without the symmetry constraint on the test function is larger than or equal
to $(n+1)/4$, which is $>1/2$ as soon as $n\geq2$.

\subsection{Linear statistics}\label{se:linear-statistics}
Recall that $\mathbb{C}^k\equiv\mathbb{R}^{2k}$ for all $k\geq1$. Fix
$g\in\mathcal{C}^1(\mathbb{C},\mathbb{C})$ such that
\begin{equation}
  \|g\|_{\mathrm{Lip}}=\sup_{z\neq z'}\frac{|g(z)-g(z')|}{|z-z'|}
  =\sup_{z}\|D g (z)\|_{\mathrm{op}}<\infty.
\end{equation}
Let $\nu^g_n$ be the image measure of $\mu_n$ by
$z\in\mathbb{C}^n\mapsto G(z)=g(z_1)+\cdots+g(z_n)\in\mathbb{C}$. For every
$F\in\mathcal{C}_c^\infty(\mathbb{C},\mathbb{R})$, the function $F\circ G$ is
symmetric and in $H^1_{\mathbb{R}}(\mu_n)$ by a standard cutoff approximation.
Moreover, for all $1\leq j\leq n$ and $z\in\mathbb{C}^n$,
$\nabla_{z_j}(F\circ G)(z)=Dg(z_j)^{\mathsf T}\nabla F(G(z))$, hence
\begin{equation}
  |\nabla(F\circ G)(z)|^2
  \leq(\|Dg(z_1)\|_{\mathrm{op}}^2+\cdots+\|Dg(z_n)\|_{\mathrm{op}}^2)
  |\nabla F(G(z))|^2
  \leq n\|g\|_{\mathrm{Lip}}^2|\nabla F(G(z))|^2.
\end{equation}
Now, the Poincaré inequality of Theorem \ref{th:main} for $f=F\circ G$ gives
\begin{equation}
  \operatorname{Var}_{\nu^g_n}(F)
  =\int_{\mathbb{C}}F^2\mathrm{d}\nu^g_n-\Bigl(\int_{\mathbb{C}}F\mathrm{d}\nu^g_n\Bigr)^2
  \leq\frac{\|g\|_{\mathrm{Lip}}^2}{2}\int_{\mathbb{C}}|\nabla F|^2\mathrm{d}\nu^g_n.
\end{equation}
Hence $\nu_n^g$ satisfies a Poincaré inequality for the Euclidean gradient
with constant $\|g\|_{\mathrm{Lip}}^2/2$. The factor $1/2$ is optimal since
if $g=\mathrm{id}$, then $G(z)=z_1+\cdots+z_n$,
$\nu^g_n=\mathcal{N}_{\mathbb{C}}(0,1)$, and $\|g\|_{\mathrm{Lip}}=1$.
% , while if $g=\Re$ or $g=\Im$, then $G=\varphi_{\mathbb{R}}$ or
% $G=\varphi_{\Im}$ respectively, while
% $\nu^g_n=\mathcal{N}_{\mathbb{R}}(0,1/2)$ and $\|g\|_{\mathrm{Lip}}=1$.

\subsection{Dynamics}

Following \cite{BolleyChafaiFontbona}, let us consider the overdamped Langevin
process associated to the Boltzmann--Gibbs measure
$\mu_n\propto\mathrm{e}^{-\beta_n E_n}$ solving the stochastic differential
equation
\begin{equation}\label{eq:intro-sde}
  \mathrm{d}Z_t^n=\sqrt{\alpha_n\frac{2}{\beta_n}}\mathrm{d}B^n_t
  -\alpha_n\nabla E_n(Z_t^n)\mathrm{d}t,
\end{equation}
where $\alpha_n>0$ is a speed parameter\footnote{Indeed ${(Z^n_t)}_{t\geq0}$
  has the law of the time-changed process ${(Y^n_{\alpha_n t})}_{t\geq0}$
  where ${(Y^n_t)}_{t\geq0}$ solves \eqref{eq:intro-sde} with $\alpha_n=1$.},
and $B^n$ is a standard Brownian motion on
$\mathbb{C}^n\equiv\mathbb{R}^{2n}$. We have
\begin{equation}
  -\nabla_{z_j}E_n(z)
  =-\frac{2}{n}z_j+\frac{2}{n^2}\sum_{k\neq j}\frac{z_j-z_k}{|z_j-z_k|^2}.
\end{equation}
We recognize an Ornstein--Uhlenbeck linear drift plus a Coulomb planar
repulsion. The coordinates $Z^{n,j}_t$, $1\leq j\leq n$ of $Z^n_t$ can be seen
as an interacting particle system consisting of $n$ confined planar Brownian
motions experiencing quadratic confinement and logarithmic or Coulomb
repulsion. Also \eqref{eq:intro-sde} is equivalent to the system of planar
stochastic differential equations
\begin{equation}
  \mathrm{d}Z^{n,j}_t
  =\sqrt{\alpha_n\frac{2}{\beta_n}}\mathrm{d}B^{n,j}_t
  -\alpha_n\Bigl(\frac{2}{n}Z^{n,j}_t+\frac{2}{n^2}\sum_{k\neq j}\frac{Z^{n,k}_t-Z^{n,j}_t}{|Z^{n,k}_t-Z^{n,j}_t|^2}\Bigr)\mathrm{d}t,
  \quad 1\leq j\leq n.
\end{equation}
The drift is singular when $Z^n$ reaches $\mathbb{C}^n_{=}$. We could use the
fact that the coefficients are in $L^2(\mu_n)$ since $1/|z_j-z_k|$ is in
$L^2(\mu_n)$, but we follow another route via Dirichlet forms. First,
following \cite{BolleyChafaiFontbona}, for every initial condition in
$\mathbb{C}^n_{\neq}$, the stochastic differential equation
\eqref{eq:intro-sde} is globally well posed and nonexplosive. Its unique
invariant probability measure is our Boltzmann--Gibbs measure $\mu_n$, and the
Markov process $Z^n$ is reversible. The Markov semigroup of $Z^n$ preserves
symmetry. The natural differential operator associated to $Z^n$ is
\begin{equation}
  A_n=\alpha_n\Bigl(\frac{1}{\beta_n}\Delta-\nabla E_n\cdot\nabla\Bigr)
  =\frac{\alpha_n}{\beta_n}\frac{\mathrm{div}(\rho_n\nabla)}{\rho_n}
  \quad\text{where}\quad
  \rho_n\propto\mathrm{e}^{-\beta_nE_n},
\end{equation}
defined on $\mathcal{C}^\infty_c(\mathbb{C}^n_{\neq},\mathbb{C})$. By
convention, we have $A_nf=A_n\Re f+\mathrm{i}A_n\Im f$. By integration by
parts,
\begin{equation}
  \langle -A_nf,g\rangle_{L^2(\mu_n)}
  =-\int_{\mathbb{C}^n}A_nf\overline{g}\,\mathrm{d}\mu_n
  =\frac{\alpha_n}{\beta_n}\int_{\mathbb{C}^n}\nabla f\cdot\overline{\nabla g}\,\mathrm{d}\mu_n
  =\mathcal{E}_n(f,g)
\end{equation}
for all $f,g\in\mathcal{C}^\infty_c(\mathbb{C}^n_{\neq},\mathbb{C})$. The
sesquilinear form $(f,g)\mapsto\mathcal{E}_n(f,g)$ is a densely defined
symmetric pre-Dirichlet form associated with the pregenerator $A_n$. Its
closure is a regular, strongly local, conservative Dirichlet form, see
\cite[Ch.~2, 4, 5, and 7]{FukushimaOshimaTakeda}, or \cite{Ouhabaz} for a more
accessible presentation. Indeed, the constants belong to $H^1(\mu_n)$ by
radial cutoff approximation, and have zero energy. The form domain is the
closure of $\mathcal{C}^\infty_c(\mathbb{C}^n_{\neq},\mathbb{C})$ for the
graph norm
\begin{equation}\label{eq:graph:norm:En}
  f\mapsto\sqrt{\|f\|_{L^2(\mu_n)}^2+\mathcal{E}_n(f,f)},
\end{equation}
which is equivalent to \eqref{eq:graphnorm} since $\alpha_n/\beta_n$ is a
constant. This closure is equal to that of
$\mathcal{C}^\infty_c(\mathbb{C}^n,\mathbb{C})$, denoted $H^1(\mu_n)$, see
Lemma \ref{le:closure}. This space $H^1(\mu_n)$ is the complexification of
$H^1_{\mathbb{R}}(\mu_n)$. Let $X^n$ be its associated Hunt diffusion. For
quasi-every initial point in $\mathbb{C}^n_{\neq}$, the process $X^n$ solves
the martingale problem for $A_n$. By Lemma \ref{le:capacity}, the set
$\mathbb{C}^n_{=}$ is polar for $X^n$.

Since $\mathcal{E}_n$ is Hermitian and non-negative, the unbounded operator
$-A_n$ in $L^2(\mu_n)$ is symmetric and nonnegative. Its Friedrichs extension
$F_n$ is the nonnegative self-adjoint operator associated, by the first
representation theorem, with the closure of $\mathcal{E}_n$. See
\cite[Ch.~VI]{Kato}.

By the correspondence between diffusion martingale problems and weak solutions
of stochastic differential equations, see \cite[Chs.~4--5]{EthierKurtz}, and
by weak uniqueness for \eqref{eq:intro-sde}, the stochastic processes $X^n$
and $Z^n$ have the same law for quasi-every initial condition in
$\mathbb{C}^n_{\neq}$. Hence, their Markov semigroups coincide as operators on
$L^2(\mu_n)$. This semigroup is ${(\mathrm{e}^{-tF_n})}_{t\geq0}$. The domain
of its infinitesimal generator $L_n=-F_n$ is
\begin{equation}
  \mathcal{D}(L_n)
  =\Bigl\{f\in H^1(\mu_n):\exists a\in L^2(\mu_n),
  \forall\varphi\in H^1(\mu_n),
  -\langle a,\varphi\rangle_{L^2(\mu_n)}=\mathcal{E}_n(f,\varphi)\Bigr\}.
\end{equation}
For all $f\in\mathcal{D}(L_n)$, the function $a\in L^2(\mu_n)$ in the
definition above is unique and we have $L_nf=a$. The operator domain
$\mathcal{D}(F_n)=\mathcal{D}(L_n)$ is smaller than the form domain
$\mathcal{D}(F_n^{1/2})=H^1(\mu_n)$.

It is natural to ask about the status of polynomials in $z$ and $\bar{z}$ as
test functions. By using a standard approximation by cutoff, it can be shown
that the algebra of polynomials $\mathbb{C}[z,\bar{z}]$ is included in
$H^1(\mu_n)$. Moreover, using the integrability properties of the density of
$\mu_n$ and integration by parts, it can be shown that
$\mathbb{C}[z,\bar{z}]\subset\mathcal{D}(L_n)$.

Following \cite[Ch.~4]{BakryGentilLedoux}, the Poincaré inequality
provided by Theorem \ref{th:main} is equivalent to stating that nonnegative
self-adjoint operator $-L_n=F_n$ has a spectral gap equal to
\begin{equation}
  2n\frac{\alpha_n}{\beta_n}=2\frac{\alpha_n}{n},
\end{equation}
when restricted to symmetric observables\footnote{In other words, for
  unlabeled particles.}. We can check by direct computation and approximation
that the center-of-mass observables $\varphi_\Re$ and $\varphi_\Im$, which are
symmetric and in $\mathcal{C}^\infty(\mathbb{C}^n,\mathbb{R})$, belong to
$\mathcal{D}(L_n)$ and are eigenfunctions of $-L_n$ associated to the spectral
gap eigenvalue, namely
\begin{equation}
  L_n\varphi_\Re=-2\frac{\alpha_n}{n}\varphi_\Re
  \quad\text{and}\quad
  L_n\varphi_\Im=-2\frac{\alpha_n}{n}\varphi_\Im.
\end{equation}
Here both the Laplacian term and the interaction term vanish due to the
linearity and symmetry.

The special choice of speed $\alpha_n=n$ gives a dimension-free spectral gap
equal to $2$.

The dependence on $n$ of the optimal Poincaré constant for planar Coulomb
dynamics, and in particular the possibility of convergence to equilibrium at a
dimension-free rate, is an old problem, see
\cite*{BolleyChafaiFontbona,LuMattingly,RosenzweigSerfatyModulatedLSI}.
Theorem~\ref{th:main} gives a sharp answer for the natural symmetric
observables of the complex Ginibre log-gas. In particular, although the
singular interaction destroys convexity, it does not increase the Poincaré
constant above the value supplied by the center-of-mass modes. The Vandermonde
transform exploits a Gaussian complex-analytic structure which is not in the
spirit of simply bounding the Hessian of the energy. We refer to
\cite{DePhilippisShenfeld} for another probabilistic functional-analytic
aspect of the complex Ginibre ensemble related to subharmonicity.

Beware that $Z^n$ is not the Dyson--Ornstein--Uhlenbeck process for complex
Ginibre matrices, and the eigenvalue dynamics is not decoupled from the
eigenvectors, see \cite{BourgadeDubachOverlaps}.

\subsection{Gaussian factorization}

The following theorem is the two-dimensional case of
\cite[Th.~6.1]{ChafaiAspects}, but the decoupling was essentially already
observed in \cite{Laughlin}. The GUE version is in \cite{ChafaiLehec}. It
reveals the presence of a Gaussian factor together with a recentered factor.
It exploits the quadratic nature of the potential as well as the translation
invariance of the Vandermonde determinant $V_n$. Indeed $V_n(z)$ depends on
the shape of the cloud of particles $z_1,\ldots,z_n$ and not on its location.
Such a factorization or splitting plays an essential role for the cutoff
phenomenon in \cite{BoursierChafaiLabbe} and \cite{ChafaiFathiCutoff}, and in
the Gibbs conditioning for large deviation principles in
\cite*{ChafaiFerreStoltz}.

\begin{theorem}[Equilibrium factorization]\label{th:split}
  Consider the complex line $L=\mathbb{C}(1,\ldots,1)\subset\mathbb{C}^n$, and
  its orthogonal complex hyperplane $L^\perp\subset\mathbb{C}^n$. Let
  $\varphi_n(z)=c_n^{-1}\mathrm{e}^{-n|z|^2}|V_n(z)|^2$, $z\in\mathbb{C}^n$,
  be the density of the complex Ginibre log-gas $\mu_n$. Then
  $z=\operatorname{proj}_L(z)+\operatorname{proj}_{L^\perp}(z)=S(z)v+W$ where
  $S(z)=z_1+\cdots+z_n$, $v=(1/n,\ldots,1/n)$ and $W=z-S(z)v$, for all
  $z\in\mathbb{C}^n$, and
  \begin{equation}
    \varphi_n(z)=\mathrm{e}^{-|S(z)|^2}\varphi_n(W).
  \end{equation}
  In probabilistic terms, if $\mathbf{Z}\sim\mu_n$ then
  $\mathbf{S}=\sum_{k=1}^n\mathbf{Z}_k$ and
  $\mathbf{W}=\mathbf{Z}-(\mathbf{S},\ldots,\mathbf{S})/n$ are independent
  with $\mathbf{S}\sim\mathcal{N}_{\mathbb{C}}(0,1)$ while $\mathbf{W}$ has
  density proportional to $\varphi_n$ for the Lebesgue measure on $L^\perp$.
  Furthermore%\footnote{If $n=1$, then $\mathbf{W}=0$ and $\mathbf{R}=0$ deterministically.}
  \[
    \mathbf{R}=n|\mathbf{W}|^2\sim\mathrm{Gamma}(\kappa,1)
    \quad\text{where}\quad
    \kappa=\frac{(n-1)(n+2)}{2}
  \]
  when $n\geq2$, while $\mathbf{S}=\mathbf{Z}$, $\mathbf{W}=0$, and
  $\mathbf{R}=0$ when $n=1$.
\end{theorem}

\begin{proof}
  A unit norm vector of $L$ is
  $u=(1/\sqrt{n},\ldots,1/\sqrt{n})\in\mathbb{C}^n$. We have
  $\operatorname{proj}_L(z)=\langle z,u\rangle u=S(z)v$. The Pythagoras
  identity gives $|z|^2=|S(z)|^2|v|^2+|z-S(z)v|^2=|S(z)|^2/n+|W|^2$. On the
  other hand, the shift invariance of $V_n$ along the line $L$ gives
  $V_n(z)=V_n(W)$ since $z-W=S(z)v\in L$. The independence follows from the
  factorization and the fact that the Jacobian is constant (linearity).

  For the last part, since $V_n$ is homogeneous of degree
  $d_n=\frac{n(n-1)}{2}$, the polar coordinates on the complex
  $(n-1)$-dimensional space $L^\perp$ give
  $\mathbf{R}=n|\mathbf{W}|^2\sim\mathrm{Gamma}(d_n+n-1,1)=\mathrm{Gamma}(\kappa,1)$.
\end{proof}

The following is the dynamical counterpart of Theorem \ref{th:split}. It is an
analogue for the overdamped Langevin process associated to the complex Ginibre
log-gas of what is done in \cite{ChafaiFathiCutoff} and
\cite{BoursierChafaiLabbe} for the Dyson--Ornstein--Uhlenbeck process
associated to the GUE.

\begin{theorem}[Dynamical factorization]
  \label{th:split:dynamical}
  Assume that $n\geq2$. Let $v$, $L$, and $L^\perp$ be as in Theorem
  \ref{th:split}. Consider the processes ${(S_t)}_{t\geq0}$ and
  ${(W_t)}_{t\geq0}$ with state space $\mathbb{C}$ and $L^\perp$ defined by
  \[
    S_tv=\operatorname{proj}_L(Z_t^n)=\Bigl(\sum_{j=1}^n Z_t^{n,j}\Bigr)v
    \quad\text{and}\quad
    W_t=\operatorname{proj}_{L^\perp}(Z_t^n)=Z_t^n-\frac{S_t}{n}(1,\ldots,1)
  \]
  where $(Z_t^n)_{t\geq0}$ is the solution of \eqref{eq:intro-sde}. Then there
  exist independent Brownian motions $B^S$ on $\mathbb{C}$ and $B^W$ on
  $L^\perp$ such that ${(S_t)}_{t\geq0}$ and ${(W_t)}_{t\geq0}$ solve the
  stochastic differential equations
  \begin{equation}\label{eq:center-relative-sde}
    \mathrm{d}S_t=\sqrt{\frac{2\alpha_n}{n}}\mathrm{d}B_t^S-\frac{2\alpha_n}{n}S_t\mathrm{d}t
    \quad\text{and}\quad
    \mathrm{d}W_t
    =
    \frac{\sqrt{2\alpha_n}}{n}\,\mathrm{d}B_t^W
    -\frac{2\alpha_n}{n}W_t\,\mathrm{d}t
    +\frac{2\alpha_n}{n^2}b(W_t)\,\mathrm{d}t,
  \end{equation}
  where $b_j(w)=\sum_{k\neq j}\frac{w_j-w_k}{|w_j-w_k|^2}$. In particular
  ${(S_t)}_{t\geq0}$ is a complex Ornstein--Uhlenbeck process and the
  transition semigroup of $Z^n$ is the product of the semigroups of
  ${(S_t)}_{t\geq0}$ and ${(W_t)}_{t\geq0}$. In particular, ${(S_t)}_{t\geq0}$
  and ${(W_t)}_{t\geq0}$ are independent whenever $S_0$ and $W_0$ are
  independent. Moreover, if we define
  \begin{equation}
    U_t=|S_t|^2,
    \quad
    R_t=n|W_t|^2
    =\sum_{1\leq j<k\leq n}|Z_t^{n,j}-Z_t^{n,k}|^2,
    \quad
    \kappa=\frac{(n-1)(n+2)}{2}.
  \end{equation}
  then there exist independent real Brownian motions $\beta^S$ and $\beta^R$
  such that ${(U_t)}_{t\geq0}$ and ${(R_t)}_{t\geq0}$ solve
  \begin{equation}
    \mathrm{d}U_t
    =
    \frac{4\alpha_n}{n}(1-U_t)\mathrm{d}t +2\sqrt{\frac{2\alpha_n}{n}U_t}\mathrm{d}\beta_t^S
    \quad\text{and}\quad
    \mathrm{d}R_t
    =
    \frac{4\alpha_n}{n}(\kappa-R_t)\mathrm{d}t +2\sqrt{\frac{2\alpha_n}{n}R_t}\mathrm{d}\beta_t^R.
  \end{equation}
  In particular ${(U_t)}_{t\geq0}$ and ${(R_t)}_{t\geq0}$ are
  Cox--Ingersoll--Ross processes with invariant distributions
  $\operatorname{Gamma}(1,1)$ and $\operatorname{Gamma}(\kappa,1)$,
  respectively. Furthermore, in the stationary case where $Z_0^n\sim\mu_n$,
  the processes $(S_t)_{t\geq0}$ and $(W_t)_{t\geq0}$, and hence
  ${(U_t)}_{t\geq0}$ and ${(R_t)}_{t\geq0}$, are independent and stationary.
\end{theorem}

\begin{proof}
  The interaction is translation invariant and satisfies
  $\sum_{j=1}^n b_j(z)=0$. Summing the coordinate equations of
  \eqref{eq:intro-sde} gives the equation for $S$. Orthogonal projection onto
  $L^\perp$ gives the equation for $W$. The corresponding projected Brownian
  motions are independent. Since the two equations have separated
  coefficients, their transition semigroups form a product. The Itô formula
  applied to $U=|S|^2$ gives the CIR equation for $U$. Applying it to
  $R=n|W|^2$ and using
  \[
    \sum_{j=1}^n
    w_j\cdot
    \sum_{k\neq j}\frac{w_j-w_k}{|w_j-w_k|^2}
    =
    \frac{n(n-1)}{2}
  \]
  gives the CIR SDE for $R$. The martingale parts $M^U$ and $M^R$ in the
  preceding Itô formulas satisfy
  \[
    \mathrm{d}\langle M^U\rangle_t = \frac{8\alpha_n}{n}U_t\,\mathrm{d}t,
    \quad
    \mathrm{d}\langle M^R\rangle_t = \frac{8\alpha_n}{n}R_t\,\mathrm{d}t,
    \quad
    \mathrm{d}\langle M^U,M^R\rangle_t=0.
  \]
  The last identity follows from the orthogonality of the Brownian motions on
  $L$ and $L^\perp$. After normalizing these martingale parts by their
  respective diffusion coefficients, and completing the definition on their
  zero sets in the usual way, the Lévy characterization yields a
  two-dimensional Brownian motion $(\beta^S,\beta^R)$. Its two components are
  independent, and the displayed CIR equations follow.
  % Following for instance \cite{BoursierChafaiLabbe}, the Brownian motions
  % $\beta$ are identified by computing the angle bracket in the Itô formula,
  % by mean of $A_nS$ and $A_nR$, the fact that $S$ and $R-\kappa$ are
  % eigenfunctions of $A_n$, and using the Lévy characterization of Brownian
  % motion among continuous local martingales.
  The invariant Gamma distributions follow from the corresponding
  one-dimensional generalized Laguerre generators
  $\frac{4\alpha_n}{n}\left(x\frac{\mathrm{d}^2}{\mathrm{d}x^2}+(\eta-x)\frac{\mathrm{d}}{\mathrm{d}x}\right)$,
  with $\eta=1$ or $\eta=\kappa$.
\end{proof}

Note that Theorem \ref{th:split} can also be seen as a corollary of Theorem \ref{th:split:dynamical}!

\subsection{Polynomial eigenfunctions}

It is convenient to write $A_n$ in symmetrized form as
\begin{equation}
  A_nf
  =\frac{\alpha_n}{\beta_n}\sum_{j=1}^n\Delta_jf
  -\alpha_n\frac{2}{n}\sum_{j=1}^nz_j\cdot\nabla_j f
  +\alpha_n\frac{2}{n^2}\sum_{j<k}\frac{1}{\bar{z}_j-\bar{z}_k}\cdot(\nabla_j-\nabla_k)f
\end{equation}
for every $f\in\mathcal{C}^\infty(\mathbb{C}^n_{\neq},\mathbb{R})$. This shows
that $A_n$ preserves symmetry. However, and in contrast with its real
counterpart, the Dyson--Ornstein--Uhlenbeck operator, considered in
\cite{ChafaiLehec}, the operator $A_n$ does not preserve symmetric polynomials
in general when $n\geq2$. A counterexample is given by the symmetric
polynomial $f(z)=z_1^2+\cdots+z_n^2$, which is holomorphic, and for which
$A_nf(z)=-\frac{4\alpha_n}{n}f+\frac{4\alpha_n}{n^2}\sum_{j<k}\frac{z_j-z_k}{\bar{z}_j-\bar{z}_k}$
is symmetric but is not a polynomial in $z$ and $\bar{z}$. Nevertheless, only
the interaction part of $A_n$ fails to preserve polynomials, and suggests
considering divisibility.

At this step, it is enlightening to rewrite $A_n$ using the complex
derivatives \eqref{eq:complex-derivatives}, namely
\begin{equation}\label{eq:An:complex}
  \frac{n^2}{2\alpha_n}A_nf
  =2
  \sum_{j=1}^n\partial_{z_j}\bar\partial_{z_j}f
  -n
  \sum_{j=1}^n \left(z_j\partial_{z_j}+\overline z_j\bar\partial_{z_j}\right)f
  +\sum_{j<k}\Bigl( \frac{(\partial_{z_j}-\partial_{z_k})f} {\overline z_j-\overline z_k}
  + \frac{(\bar\partial_{z_j}-\bar\partial_{z_k})f} {z_j-z_k} \Bigr).
\end{equation}
This suggests that $A_nf$ is a polynomial when $f$ is a polynomial and
$(\partial_{z_j}-\partial_{z_k})f$ is divisible by $\bar{z}_j-\bar{z}_k$ and
$(\bar\partial_{z_j}-\bar\partial_{z_k})f$ is divisible by $z_j-z_k$, for all $j<k$.
This leads to the following theorem.

\begin{theorem}[Polynomial eigenfunctions]\label{th:poly}
  Assume that $n\geq2$, and consider the following:
  \begin{itemize}
  \item Two symmetric functions, which are polynomials with respect to $z$ and $\bar{z}$:
    \begin{equation}\label{eq:S-R}
      S(z)=\sum_{j=1}^n z_j
      \quad\text{and}\quad
      R(z)=\sum_{1\leq j<k\leq n}|z_j-z_k|^2=n\sum_{j=1}^n\left|z_j-\frac{S(z)}n\right|^2.
    \end{equation}
  \item The normalized complex Hermite polynomial\footnote{The Rodrigues
      formula is
      $\mathrm H_{a,b}(s,\overline
      s)=\frac{(-1)^{a+b}}{\sqrt{a!b!}}\mathrm{e}^{|s|^2}\bar\partial_s^{a}\partial_s^{b}\mathrm{e}^{-|s|^2}$.}$\mathrm
    H_{a,b}$ of bidegree $(a,b)$ for $\gamma_1$.
  \item The classical generalized Laguerre polynomial\footnote{Defined by the
      Laguerre equation \eqref{eq:laguerre} and normalized by
      $L^{\alpha}_m(0)=\binom{m+\alpha}{m}$. These polynomials are orthogonal
      with respect to the probability distribution
      $\mathrm{Gamma}(\kappa,1)=\frac{1}{\Gamma(\kappa)}x^{\kappa-1}\mathrm{e}^{-x}\mathbf{1}_{x\geq0}\mathrm{d}x$.
      However, they are not orthonormal.} $L_m^{(\kappa-1)}$ of degree $m$,
    with $\kappa=\frac{(n-1)(n+2)}{2}$.
  \end{itemize}
  Then for every $a,b,m\in\mathbb N$, the symmetric polynomial
  \begin{equation}\label{eq:eigenpolynomial}
    P_{a,b,m}(z)
    =\mathrm H_{a,b}\bigl(S(z),\overline{S(z)}\bigr)
    L_m^{(\kappa-1)}\bigl(R(z)\bigr)
  \end{equation}
  is an eigenfunction of $A_n$ in the sense that
  \begin{equation}\label{eq:eigenvalue}
    -A_nP_{a,b,m}  =\lambda_{a,b,m}P_{a,b,m}
    \quad\text{where}\quad
    \lambda_{a,b,m}=2\frac{\alpha_n}{n}(a+b+2m).
  \end{equation}
  Moreover, the family $\{P_{a,b,m}:a,b,m\in\mathbb{N}\}$ is a basis of the
  polynomial algebra $\mathbb C[S,\overline S,R]$. In particular, this algebra
  is invariant under $A_n$ in the sense that
  $A_n\bigl(\mathbb C[S,\overline S,R]\bigr)\subset\mathbb C[S,\overline
  S,R]$. Furthermore, the elements of the family are mutually orthogonal in
  $L^2(\mu_n)$.
\end{theorem}

A proof of Theorem \ref{th:poly} is given in Section \ref{se:proof:th:poly}.

Here are, for instance, the first seven polynomial eigenfunctions provided by Theorem \ref{th:poly}:
\begin{equation}
  P_{0,0,0}=1,\ %
  P_{1,0,0}=S,\ %
  P_{0,1,0}=\overline{S},\ %
  P_{0,0,1}=\kappa-R,%
\end{equation}
\begin{equation}
  P_{2,0,0}=\frac{S^2}{\sqrt{2}},\ %
  P_{1,1,0}=|S|^2-1,\ %
  P_{0,2,0}=\frac{\overline{S}^2}{\sqrt{2}}.
\end{equation}

A standard approximation argument by cutoff shows that all the polynomials
provided by Theorem \ref{th:poly} belong to $H^1(\mu_n)$, and by integration
by parts, to $\mathcal{D}(L_n)$. This provides the following corollary.

\begin{corollary}[Eigenvalues and eigenfunctions of $L_n$]
  The spectrum of $L_n$ contains
  \begin{equation}
    \Bigl\{-2\frac{\alpha_n}{n}k:k\in\mathbb{N}\Bigr\}.
  \end{equation}
\end{corollary}

The self-adjointness of $L_n$ implies automatically that the eigenfunctions
associated to distinct eigenvalues are orthogonal in $L^2(\mu_n)$. But Theorem
\ref{th:poly} provides more orthogonality.

\begin{remark}[Incompleteness]
  The polynomial eigenfunctions in Theorem \ref{th:poly} are orthogonal in
  $L^2(\mu_n)$, but the closure of their span does not give all symmetric
  elements of $L^2(\mu_n)$. For instance
  $Q(z)=\sum_{j=1}^n(z_j-\frac{S(z)}{n})^2$ is a non-zero symmetric polynomial
  which is orthogonal in $L^2(\mu_n)$ to $\mathbb{C}[S,\overline{S},R]$.
  Indeed, if we define
  $T_\theta(z)=S(z)v+\mathrm{e}^{\mathrm{i}\theta}(z-S(z)v)$ where
  $v=(1/n,\ldots,1/n)$. Then $S(T_\theta(z))=S(z)$, $R(T_\theta(z))=R(z)$,
  whereas $Q(T_\theta(z))=\mathrm{e}^{2\mathrm{i}\theta}Q(z)$. For every
  $F\in\mathbb{C}[S,\overline{S},R]$, $F\circ T_\theta=F$,
  $Q\circ T_\theta=\mathrm{e}^{2\mathrm{i}\theta}Q$. Since $T_\theta$
  preserves $\mu_n$, we get
  $\langle Q,F\rangle_{L^2(\mu_n)}=\mathrm{e}^{2\mathrm{i}\theta}\langle
  Q,F\rangle_{L^2(\mu_n)}$ for all $\theta$. Hence $Q\perp F$. Note that
  $\sum_{j=1}^nz_j^2=\frac{S^2}{n}+Q$. From this point of view, Theorem
  \ref{th:poly} can only be seen as a partial counterpart of
  \cite{LassalleHermite,LassalleLaguerre} for the complex Ginibre log-gas.
\end{remark}

The study of the polynomial eigenfunctions of reversible Markov diffusion
operators is a fascinating subject. We refer to
\cite{BakryPolynomialEigenvectors,BakryOrevkovZani}. The differential operator
can then be seen as the multivariate ordinary differential equation associated
to the family of orthogonal polynomials with respect to the equilibrium, see
for instance
\cite{LassalleHermite,LassalleJacobi,LassalleLaguerre,BakerForresterCalogeroSutherland}.

\subsection{Curvatures and deficits}
\label{se:BE}

In this subsection, we take $\alpha_n=n$ for $A_n$ to be aligned with Theorem
\ref{th:main}. The overdamped Langevin operator $A_n$ can then be written on
$\mathbb{C}^n_{\neq}$ as
\begin{equation}\label{eq:generator-alpha-n}
  A_n
  =
  \frac{1}{n}
  \left(
    \Delta-\nabla\Phi_n\cdot\nabla
  \right)
  \quad\text{where}\quad
  \Phi_n(z)
  =n^2E_n(z)
  =n\sum_{j=1}^n|z_j|^2
  -
  2\sum_{j<k}\log|z_j-z_k|.
\end{equation}
The $\Gamma$ and $\Gamma_2$ functional quadratic forms of $A_n$ are given, for all
$f\in\mathcal{C}^\infty_c(\mathbb{C}^n_{\neq},\mathbb{R})$, by
\begin{equation}\label{eq:gamma-two-bochner}
  \Gamma(f)=\Gamma(f,f)=\frac{1}{n}|\nabla f|^2
  \quad\text{and}\quad
  \Gamma_2(f)=\Gamma_2(f,f)
  =\frac{1}{n^2}\bigl(\|\nabla^2f\|_{\mathrm{HS}}^2
  +\langle\nabla^2\Phi_n\nabla f,\nabla f\rangle\bigr).
\end{equation}
The formula for $\Gamma_2$ is a consequence of the Bochner commutation formula
$\nabla A_n=A_n\nabla-\frac{1}{n}\nabla^2\Phi_n\nabla$. Moreover, invariance and
reversibility of $\mu_n$ give, by integration by parts, for all
$f\in\mathcal{C}^\infty_c(\mathbb{C}^n_{\neq},\mathbb{R})$,
\begin{equation}
  \int_{\mathbb{C}^n}\Gamma(f)\mathrm{d}\mu_n =\mathcal{E}_{\mu_n}(f,f)
  \quad\text{and}\quad
  \int_{\mathbb{C}^n}\Gamma_2(f)\mathrm{d}\mu_n =\int_{\mathbb{C}^n}(A_nf)^2\mathrm{d}\mu_n,
\end{equation}
see \cite{BakryPolynomialEigenvectors} or \cite{AneEtAlLSI}. In this context,
the Bakry--Émery curvature is simply the largest constant for which
$\Gamma_2\geq\rho\Gamma$, in other words the infimum over
$\mathbb{C}^n_{\neq}$ of the smallest eigenvalue of the Hessian matrix
$\frac{1}{n}\nabla^2\Phi_n$. The following lemma is taken from
\cite{BolleyChafaiFontbona}.

\begin{lemma}[Lack of lower bound on curvature]
  The Bakry--Émery curvature is $-\infty$.
\end{lemma}

\begin{proof} 
  Let us separate confinement and interaction terms as
  \begin{equation}
    \Phi_n(z) = n|z|^2+\Phi_n^{\mathrm{int}}(z),
    \quad\text{where}\quad
    \Phi_n^{\mathrm{int}}(z) = -2\sum_{j<k}\log|z_j-z_k|.
  \end{equation}
  In the relative variable $u=z_j-z_k$,
  \[
    \nabla^2\bigl(-2\log|u|\bigr)
    =
    \frac{2}{|u|^2}
    \left(
      2\frac{u}{|u|}\otimes\frac{u}{|u|}
      -
      I_2
    \right),
  \]
  whose radial and tangential eigenvalues are respectively $2/|u|^2$ and
  $-2/|u|^2$. In particular, the interaction curvature is unbounded from below
  as $u$ tends to $0$. More precisely, consider a path along which $z_j-z_k$
  tends to zero while all other pairwise differences remain bounded away from
  zero, and take a unit relative direction tangential to $z_j-z_k$. The
  contribution of the pair $(j,k)$ tends to $-\infty$, whereas the confinement
  term and all other interaction contributions remain bounded. Hence the
  smallest eigenvalue of $\frac{1}{n}\nabla^2\Phi_n$ tends to $-\infty$ when
  approaching $\mathbb{C}^n_{=}$, the edge of $\mathbb{C}^n_{\neq}$.
\end{proof}

\begin{remark}[Constant mean curvature]
  On $\mathbb{C}^n_{\neq}$, the interaction is harmonic, thus
  $\Delta\Phi_n^{\mathrm{int}}=0$. Since the ambient real dimension is $2n$,
  we get $\Delta\Phi_n=4n^2$. Thus the normalized trace of the Bakry--Émery
  curvature matrix is
  \begin{equation}\label{eq:mean-curvature-two}
    \frac{1}{2n}\operatorname{Tr}\Bigl(\frac{1}{n}\nabla^2\Phi_n\Bigr)=2
    \quad\text{ on }\mathbb{C}^n_{\neq}.
  \end{equation}
  In this limited sense, the mean curvature is equal to $2$. This trace
  identity does not imply either the pointwise inequality
  $\Gamma_2\geq2\Gamma$ or its integrated version. It is tempting to ask if
  after factorizing the center-of-mass Ornstein--Uhlenbeck process, the
  remaining recentered process has then a plain Bakry--Émery curvature equal
  to $2$. Unfortunately, this is not the case when $n\geq2$. The curvature of
  the recentered process remains equal to $-\infty$. Curvature is a second
  order notion. The obstruction is the local tangential direction near a pair
  collision, which belongs to $L^\perp$.
  % related to the fact that $\mathbb{C}^n_{\neq}$ has dimension $\geq2$.
\end{remark}

Following \cite{BakryGentilLedoux}, the Poincaré inequality for $\mu_n$ with
constant $1/2$ in Theorem \ref{th:main} is equivalent to the integrated
Bakry--Émery $\Gamma_2$ criterion with curvature $2$:
\begin{equation}\label{eq:gamma2int}
  \int_{\mathbb{C}^n}(A_nf)^2\mathrm{d}\mu_n
  \geq\frac{2}{n}\int_{\mathbb{C}^n}|\nabla f|^2\mathrm{d}\mu_n
  \quad\text{in other words}\quad 
  \|A_nf\|_{L^2(\mu_n)}^2\geq2\mathcal{E}_{\mu_n}(f,f),
\end{equation}
for all symmetric $f\in\mathcal{C}^\infty_c(\mathbb{C}^n_{\neq},\mathbb{R})$.
From \eqref{eq:gamma-two-bochner} we get that for all
$f\in\mathcal{C}^\infty_c(\mathbb{C}^n_{\neq},\mathbb{R})$,
\begin{equation}\label{eq:gamma-two-deficit}
  \Gamma_2(f)-2\Gamma(f)
  =\frac{1}{n^2}\Bigl(\|\nabla^2f\|_{\mathrm{HS}}^2
  +\langle\nabla^2\Phi_n^{\mathrm{int}}\nabla f,\nabla f\rangle\Bigr).
\end{equation}
The right-hand side is not pointwise nonnegative. Nevertheless, Theorem
\ref{th:main} says precisely that its integral with respect to $\mu_n$ is
nonnegative when $f$ is symmetric. This suggests seeking an expression of the
integral of the right-hand side of \eqref{eq:gamma-two-deficit} as a sum of
squares, when $f$ is symmetric.

In the following couple of theorems, the map $U$ and the set $\mathcal{H}_n$ are as in Section
\ref{se:proof:th:main}.

\begin{theorem}[Sum of squares for the Poincaré and $\Gamma_2$ deficits]
  \label{th:deficit1}
  For all symmetric $f\in\mathcal{C}^\infty_c(\mathbb{C}^n_{\neq},\mathbb{R})$, 
  \begin{align}
    \mathcal{E}_{\mu_n}(f,f)
    -2\operatorname{Var}_{\mu_n}(f)
    &=2\|r\|_{L^2(\mu_n)}^2
    +4\sum_{m=2}^{\infty}(m-1)\|g_m\|_{L^2(\gamma_n)}^2\label{eq:poincare-deficit-sum-of-squares}\\
    \|A_nf\|_{L^2(\mu_n)}^2-2\mathcal{E}_{\mu_n}(f,f)
    &=
    \|(A_n+2)\widetilde f\|_{L^2(\mu_n)}^2
    +4\|r\|_{L^2(\mu_n)}^2
    +8\sum_{m=2}^{\infty}(m-1)\|g_m\|_{L^2(\gamma_n)}^2\label{eq:integrated-gamma-two-sum-of-squares}
  \end{align}
  where
  \begin{itemize}
  \item $\widetilde f=f-\int_{\mathbb{C}^n}f\mathrm{d}\mu_n$
  \item $g=\sum_{m=0}^\infty g_m$ is the Hermite expansion of
    $g=U\widetilde f$ in $L^2(\gamma_n)$, with $m=|q|$
  \item $r=\widetilde f-h-\overline{h}$
  \item
    $h=U^{-1}g_0=U^{-1}\operatorname{proj}_{\mathcal{H}_n}(U\widetilde{f})$,
    $\overline{h}$ denotes the pointwise complex conjugate of $h$
  \end{itemize}
\end{theorem}

The proof of Theorem \ref{th:deficit1} is given in Section \ref{se:proof:th:deficit}.

The following theorem provides an alternative complex analytic way to express the deficits.

\begin{theorem}[Sum of squares for the Poincaré and $\Gamma_2$ deficits]
  \label{th:deficit2}
  For all symmetric
  $f\in\mathcal{C}^\infty_c(\mathbb{C}^n_{\neq},\mathbb{R})$,
  \begin{align}
    \mathcal{E}_{\mu_n}(f,f)-2\operatorname{Var}_{\mu_n}(f)
    &=2\|r\|_{L^2(\mu_n)}^2
      +\frac{4}{n}\sum_{j,k=1}^n\|\bar\partial_{z_j}\bar\partial_{z_k}v\|_{L^2(\gamma_n)}^2
      \label{eq:poincare-deficit-differential-squares}\\
    \|A_nf\|_{L^2(\mu_n)}^2-2\mathcal{E}_{\mu_n}(f,f)
    &=\|(A_n+2)\widetilde f\|_{L^2(\mu_n)}^2
      +4\|r\|_{L^2(\mu_n)}^2
      +\frac{8}{n}\sum_{j,k=1}^n\|\bar\partial_{z_j}\bar\partial_{z_k}v\|_{L^2(\gamma_n)}^2
      \label{eq:gamma-two-deficit-differential-squares}
  \end{align}    
  where
  \begin{itemize}
  \item
    $v =
    \mathcal{N}_n^{-1/2}\operatorname{proj}_{\mathcal{H}_n^\perp}(U\widetilde
    f)$ where
    $\mathcal{N}_n
    =\sum_{j=1}^n\mathcal{N}_{\bar{z}_j}=\sum_{j=1}^n(-\partial_{z_j}+n\overline
    z_j)\bar\partial_{z_j}$.
  \item $\widetilde f$, $h$, $\overline h$, and $r$ are as in Theorem \ref{th:deficit1}
  \end{itemize}
\end{theorem}

%The formula is global and nonlocal through $\operatorname{proj}_{\mathcal{H}_n}$ and $\mathcal{N}_n^{-1/2}$.

The proof of Theorem \ref{th:deficit2} is given in Section \ref{se:proof:th:deficit}.

\subsection{Log-Sobolev inequalities}

Given Theorem \ref{th:main}, it is quite natural to ask if for all $n\geq1$,
there exists a constant $0<C_n<\infty$ such that for all
$f\in\mathcal{C}^\infty_c(\mathbb{C}^n,\mathbb{R})$,
\begin{equation}\label{eq:LSI}
  \operatorname{Ent}_{\mu_n}(f^2)\leq C_n\mathcal{E}_{\mu_n}(f,f),
\end{equation}
where $\mathcal{E}_{\mu_n}(f,f)$ is as in Theorem \ref{th:main}, and where, using the convention $0\log(0)=0$,
\begin{equation}
  \operatorname{Ent}_{\mu_n}(f^2)
  =\int_{\mathbb{C}^n}f^2\log(f^2)\mathrm{d}\mu_n
  -\Bigl(\int_{\mathbb{C}^n}f^2\mathrm{d}\mu_n\Bigr)\log\int_{\mathbb{C}^n}f^2\mathrm{d}\mu_n.
\end{equation}
We can pass from $f\in\mathcal{C}^\infty_c(\mathbb{C}^n,\mathbb{R})$ to
$H^1_{\mathbb{R}}(\mu_n)$ by standard approximation methods. Following
\cite{Gross}, the inequality \eqref{eq:LSI} is known as a logarithmic Sobolev
inequality of constant $C_n$ for functions in
$\mathcal{C}^\infty_c(\mathbb{C}^n,\mathbb{R})$. The linearization
$f=1+\varepsilon g$ allows to extract from \eqref{eq:LSI} a Poincaré
inequality with constant $C_n/2$ for $g$. Together with Theorem \ref{th:main},
this leads to formulate the following problem.

\begin{problem}[Logarithmic Sobolev inequalities]\label{pb:lsi}
  Find the sharp (best or smallest) constant for the logarithmic Sobolev
  inequality \eqref{eq:LSI} for $\mu_n$, for symmetric functions. Does it
  depend on $n$? Is it equal to $1$? Alternatively, find a counterexample
  (sequence of functions). Possibly, further specialize the problem to linear
  statistics, as done in Section \ref{se:linear-statistics} for the Poincaré
  inequality.
\end{problem}

We give now an answer for symmetric radial functions, based on the moduli
decoupling in $\mu_n$ which goes back to \cite{Kostlan}. The constant is
dimension-free, as in Theorem \ref{th:main}.

\begin{theorem}[Uniform logarithmic Sobolev inequality for symmetric radial functions]\label{th:rlsi}
  For all $n\geq1$, and all symmetric
  $F\in\mathcal{C}^\infty_c((0,\infty)^n,\mathbb{R})$, the function
  $f(z_1,\ldots,z_n)=F(|z_1|,\ldots,|z_n|)$ satisfies
  \begin{equation}\label{eq:rlsi}
    \operatorname{Ent}_{\mu_n}(f^2)
    \leq
    %\frac{1}{n}
    %\int_{\mathbb{C}^n}|\nabla f|^2\,\mathrm{d}\mu_n
    %=
    \mathcal{E}_{\mu_n}(f,f).
  \end{equation}
  Moreover, this remains valid for the $H^1_{\mathbb{R}}(\mu_n)$-closure of
  such smooth symmetric radial functions.
\end{theorem}

The proof of Theorem \ref{th:rlsi} is given in Section \ref{se:proof:th:rlsi}.

We could think of proving Theorem \ref{th:rlsi} using the Gaussian optimal
logarithmic Sobolev inequality for complex Ginibre random matrices.
Unfortunately, for such non-normal matrices, the eigenvalues map is not
Lipschitz. Still we can extract something using this matrix lifting approach.

Let $\gamma_n^{\mathrm{mat}}$ denote the law on
$\mathbb{C}^{n\times n}\equiv\mathbb{R}^{2n^2}$ of a random matrix $G_n$
whose entries are independent with distribution
$\mathcal{N}_{\mathbb{C}}(0,1/n)$. Thus $\gamma_n^{\mathrm{mat}}$ has density
proportional to
\[
  G\longmapsto\mathrm{e}^{-n\operatorname{Tr}(GG^*)}.
\]
The spectrum of $G_n$ is almost surely simple, that is, has pairwise distinct
eigenvalues. On the simple-spectrum set, choose any measurable labeling
$\lambda_1(G),\ldots,\lambda_n(G)$ of the eigenvalues, and choose right and
left eigenvectors $r_j(G)$ and $\ell_j(G)$ normalized by
\begin{equation}\label{eq:left-right-eigenvectors}
  Gr_j=\lambda_jr_j,
  \qquad
  \ell_j^*G=\lambda_j\ell_j^*,
  \qquad
  \ell_j^*r_k=\delta_{jk}.
\end{equation}
The matrices $\Pi_j=r_j\ell_j^*$ are the spectral projectors of $G$. Define
the overlap Gram matrix
\begin{equation}\label{eq:overlap-gram}
  \mathcal{O}_{jk}(G) = \operatorname{Tr}(\Pi_k\Pi_j^*),\quad 1\leq j,k\leq n.
\end{equation}
It is Hermitian and nonnegative, since
\begin{equation}\label{eq:overlap-positive}
  c^*\mathcal{O}(G)c = \Bigl\| \sum_{j=1}^n c_j\Pi_j \Bigr\|_{\mathrm{HS}}^2 \geq0.
\end{equation}
Moreover, $\mathcal{O}(G)$ does not depend on the normalization of the left
and right eigenvectors. We refer to
\cite{BourgadeDubachOverlaps} for the study of these overlaps in the complex
Ginibre ensemble.

\begin{theorem}[Functional inequalities from the matrix lift and eigenvector
  overlaps]\label{th:mlsi}
  For every symmetric $f\in\mathcal{C}^1(\mathbb{C}^n,\mathbb{R})$ such that
  $\Phi_f(G)=f(\lambda_1(G),\ldots,\lambda_n(G))\in
  H^1(\gamma_n^{\mathrm{mat}})$, we have
  \begin{equation}
    \operatorname{Var}_{\mu_n}(f)
    \leq
    \frac{2}{n}
    \mathbb{E}\left[
      a_f(G_n)^*\mathcal{O}(G_n)a_f(G_n)
    \right]
    \text{ and }
    \operatorname{Ent}_{\mu_n}(f^2)
    \leq
    \frac{4}{n}
    \mathbb{E}\left[
      a_f(G_n)^*\mathcal{O}(G_n)a_f(G_n)
    \right]
  \end{equation}
  where $a_f(G)$ is the random vector of $\mathbb{C}^n$ given by
  $(a_f(G))_j=\partial_{z_j}f(\lambda_1(G),\ldots,\lambda_n(G))$. Moreover the
  quadratic form on the right-hand sides is independent of the labeling of the
  eigenvalues.
\end{theorem}

Theorem \ref{th:mlsi} produces a Dirichlet form weighted by eigenvector
overlaps. It does not by itself give a comparison with $\mathcal{E}_{\mu_n}$,
since such a comparison would require an additional estimate of
$a_f(G_n)^*\mathcal{O}(G_n)a_f(G_n)$ in terms of $\sum_j|\partial_{z_j}f|^2$.
Thus the matrix lift alone does not solve Problem \ref{pb:lsi}. This is in
contrast with the situation for GUE, for which the overlap matrix is the
identity due to the normality (not in the sense of Gaussianity!) of the
matrix.
% The Sobolev assumption on $\Phi_f$ should not be removed without a separate
% analysis near the multiple-spectrum set.

\begin{proof}
  Let $G$ have simple spectrum. First-order perturbation theory gives, for
  every $H\in\mathbb{C}^{n\times n}$,
  \begin{equation}\label{eq:eigenvalue-differential}
    \mathrm{D}\lambda_j(G)[H]
    =
    \ell_j^*Hr_j
    =
    \operatorname{Tr}(\Pi_jH).
  \end{equation}
  Since $f$ is real-valued,
  \begin{align}
    \mathrm{D}\Phi_f(G)[H]
    &=
    2\operatorname{Re}
    \sum_{j=1}^n
    \partial_{z_j}f(\lambda_1(G),\ldots,\lambda_n(G))
    \operatorname{Tr}(\Pi_jH).
  \end{align}
  Equip $\mathcal{M}_n(\mathbb{C})=\mathbb{C}^{n\times n}$ with the real
  Hilbert--Schmidt scalar product
  \[
    \langle K,H\rangle_{\mathrm{HS},\mathbb{R}}
    =
    \operatorname{Re}\operatorname{Tr}(K^*H).
  \]
  It follows that
  $\nabla\Phi_f(G)=2\sum_{j=1}^n\overline{\partial_{z_j}f(\lambda_1(G),\ldots,\lambda_n(G))}\Pi_j^*$,
  and thus
  \begin{equation}\label{eq:matrix-lift-gradient}
    \|\nabla\Phi_f(G)\|_{\mathrm{HS}}^2
    =
    4a_f(G)^*\mathcal{O}(G)a_f(G).
  \end{equation}
  In real coordinates, the probability distribution $\gamma_n^{\mathrm{mat}}$
  is the centered Gaussian measure with covariance matrix $(2n)^{-1}I_{2n^2}$.
  Its Gaussian Poincaré inequality, see
  \cite[Ch.~5]{BakryPolynomialEigenvectors}, gives
  \begin{equation}
    \operatorname{Var}_{\gamma_n^{\mathrm{mat}}}(\Phi_f)
    \leq
    \frac{1}{2n}
    \int_{\mathbb{C}^{n\times n}}
    \|\nabla\Phi_f(G)\|_{\mathrm{HS}}^2
    \,\mathrm{d}\gamma_n^{\mathrm{mat}}(G),
  \end{equation}
  while its Gaussian logarithmic Sobolev inequality gives
  \begin{equation}
    \operatorname{Ent}_{\gamma_n^{\mathrm{mat}}}(\Phi_f^2)
    \leq
    \frac{1}{n}\int_{\mathbb{C}^{n\times n}}
    \|\nabla\Phi_f(G)\|_{\mathrm{HS}}^2
    \mathrm{d}\gamma_n^{\mathrm{mat}}(G).
  \end{equation}
  Since $f$ is symmetric, $\Phi_f$ is independent of the chosen labeling.
  The unordered spectrum of $G_n$ has the distribution induced by $\mu_n$.
  Hence
  \begin{equation}
    \operatorname{Var}_{\gamma_n^{\mathrm{mat}}}(\Phi_f)
    =\operatorname{Var}_{\mu_n}(f)
    \quad\text{and}\quad
    \operatorname{Ent}_{\gamma_n^{\mathrm{mat}}}(\Phi_f^2)
    =\operatorname{Ent}_{\mu_n}(f^2).
  \end{equation}
  Combining these identities with \eqref{eq:matrix-lift-gradient} proves the
  desired inequalities.

  Finally, under a permutation of the eigenvalue labels, the entries of
  $a_f$ and the rows and columns of $\mathcal{O}$ are permuted in the same
  way. The quadratic form $a_f^*\mathcal{O}a_f$ is therefore unchanged.
\end{proof}

\subsection{Non-quadratic potential}

\begin{theorem}[Determinantal log-gases with non-quadratic confining potential]
  \label{th:nonquad}
  Take $n\geq1$ and let
  $V\in\mathcal{C}^2(\mathbb{C}\equiv\mathbb{R}^2,\mathbb{R})$ be rotationally
  invariant, and such that
  \begin{equation}
    Z_n^V=\int_{\mathbb{C}^n}\mathrm{e}^{-n\sum_{k=1}^nV(z_k)}\prod_{j<k}|z_j-z_k|^2\mathrm{d}z_1\cdots\mathrm{d}z_n <\infty.
  \end{equation}
  Consider the Boltzmann--Gibbs probability measure on $\mathbb{C}^n$ defined by
  \begin{equation}
    \mu_n^V=\frac{1}{Z_n^V}\mathrm{e}^{-n\sum_{k=1}^nV(z_k)}\prod_{j<k}|z_j-z_k|^2\mathrm{d}z_1\cdots\mathrm{d}z_n.
  \end{equation}
  Then for every constant $\rho>0$, the following properties are satisfied.
  \begin{enumerate}
  \item If $V$ is $\rho$-subharmonic, in the sense that
    $V-\frac{\rho}{2}\left|\cdot\right|^2$ is subharmonic, then $\mu_n^V$
    satisfies a Poincaré inequality for symmetric functions: for all symmetric
    $f\in\mathcal{C}^\infty_c(\mathbb{C}^n,\mathbb{R})$,
    \begin{equation}
      \operatorname{Var}_{\mu_n^V}(f)%
      \leq\frac{1}{\rho n}\int_{\mathbb{C}^n}|\nabla f|^2\mathrm{d}\mu_n^V.
    \end{equation}
  \item If $V$ is $\rho$-convex, in the sense that
    $V-\frac{\rho}{2}\left|\cdot\right|^2$ is convex, then $\mu_n^V$ satisfies
    a logarithmic Sobolev inequality for symmetric radial functions: for all
    symmetric $F\in\mathcal{C}^\infty_c((0,\infty)^n,\mathbb{R})$, the
    function $f(z_1,\ldots,z_n)=F(|z_1|,\ldots,|z_n|)$ satisfies
    \begin{equation}
      \operatorname{Ent}_{\mu_n^V}(f^2)%
      \leq\frac{2}{\rho n}\int_{\mathbb{C}^n}|\nabla f|^2\mathrm{d}\mu_n^V.
    \end{equation}
  \end{enumerate}
\end{theorem}

The constants may not be sharp and the optimal functions are unknown. For
$\rho=2$ and $V=\left|\cdot\right|^2$, we recover the Poincaré inequality of
Theorem \ref{th:main} and the log-Sobolev inequality of Theorem \ref{th:rlsi}.

The first property is reminiscent of \cite{DePhilippisShenfeld}.

\begin{proof}
  Follows by a careful reading of the proof of Theorem \ref{th:main} in
  Section \ref{se:proof:th:main} and of Theorem \ref{th:rlsi} in Section
  \ref{se:proof:th:rlsi}. Indeed, for the Poincaré inequality, the Hermite
  proof of the Gaussian $\bar\partial$ Poincaré inequality of Lemma
  \ref{le:dbar-gap}, which is purely Gaussian, can be replaced by the
  Hörmander--Berndtsson estimate discussed in Remark \ref{rm:horber}, with gap
  $n\rho/2$. The rotational invariance of $V$ is needed for the
  holomorphic--antiholomorphic orthogonality in Lemma
  \ref{le:projection-geometry}, whose proof applies verbatim to $\mu_n^V$.

  For the log-Sobolev inequality, following
  \cite[Th.~26]{HoughKrishnapurPeresViragIndependence}, the Kostlan
  observation remains available since we deal with a planar determinantal
  point process of $n$ particles with a radial potential, and we get this
  time,
  denoting $V(z)=Q(|z|)$, the probability measure
  \begin{equation}
    \nu^V_{n,k}=\frac{1}{Z^V_{n,k}}r^{2k-1}
    \mathrm{e}^{-nQ(r)}\mathbf{1}_{r\geq0}\mathrm{d}r,
  \end{equation}
  see also \cite{zbMATH06347751}. This is the law of $|X_k|$ where $X_k$ is a
  random vector of $\mathbb{R}^{2k}$ with density proportional to
  $\mathrm{e}^{-nQ(\left|\cdot\right|)}$. Since $V$ is radial and
  $V-\frac{\rho}{2}\left|\cdot\right|^2$ is convex on $\mathbb{R}^2$, the
  function $Q(\left|\cdot\right|)-\frac{\rho}{2}\left|\cdot\right|^2$ is
  convex on $\mathbb{R}^{2k}$. Indeed, the convexity of
  $\psi(\left|\cdot\right|)$ on $\mathbb{R}^2$, where
  $\psi(r)=Q(r)-\frac{\rho}{2}r^2$, implies that $\psi$ is convex and
  nondecreasing on $[0,\infty)$, hence $\psi(\left|\cdot\right|)$ is convex in
  any Euclidean space. Therefore, the Bakry--Émery curvature
  criterion\footnote{Alternatively, the Caffarelli contraction theorem, or the
    Cordero--Erausquin transport theorem.}, see \cite{BakryGentilLedoux},
  gives that the law of $X_k$ satisfies a log-Sobolev inequality with constant
  $2/(n\rho)$, for the pure Euclidean gradient. The same holds for $\nu^V_{n,k}$
  above since $x\mapsto|x|$ is $1$-Lipschitz. The remaining tensorization
  argument and gradient transformation are completely the same as in the
  quadratic potential case.
\end{proof}

\subsubsection{Open problems}

The personal interest of the author in these questions emerged during the
preparation of \cite*{zbMATH06371853}, see also
\cite{RosenzweigSerfatyModulatedLSI} for another side. At least for linear
statistics, they make sense given the available central limit theorems.

\begin{problem}[Arbitrary inverse temperatures]
  % The proofs of Theorem \ref{th:main} and Theorem \ref{th:rlsi} rely on the
  % fact that $\beta_n=n^2$ in \eqref{eq:boltzmanngibbs}.
  How about Poincaré and log-Sobolev inequalities for the log-gas obtained
  when replacing $\beta_n=n^2$ by an arbitrary positive value in
  \eqref{eq:boltzmanngibbs}? Dimension-free log-Sobolev inequalities when
  $\beta_n$ is proportional to $n$ are obtained in \cite{rosenzweig}.
\end{problem}

\begin{problem}[Higher-dimensional log-gases]
  How about Poincaré and log-Sobolev inequalities for high-dimensional
  Riesz/Coulomb log-gases? Uniform Poincaré inequalities for mean-field
  Coulomb gases in dimension $d\geq3$ under weak coupling are obtained in
  \cite{becker-menegaki}.
\end{problem}

\subsection{About this article}

The author completed the first version of this article in January 2026, and
decided to explore related questions before making it public. Although the
author wrote and independently checked the final manuscript, the ChatGPT AI
was used for mathematical exploration and proofreading, and the Gemini AI was
used as a second proofreading tool. The mathematical motivation to study the
complex Ginibre log-gas came almost twenty years ago, while discovering random
matrix theory with a culture in statistical mechanics and probabilistic
functional analysis. First attempts were made in \cite*{zbMATH06371853},
\cite*{BolleyChafaiFontbona}, and \cite*{ChafaiHardyMaida}. The author had
already in mind the statement of Theorem \ref{th:main}, which is the direct
analogue of the GUE case in \cite{ChafaiLehec}, and the possible usage of
holomorphicity and Hörmander--Berndtsson theory as in
\cite{CorderoErausquinBerndtssonPrekopa}. However, the final idea of the
decomposition used in the proof of Theorem \ref{th:main} is due to ChatGPT.
Seven months after completing this proof, the author became aware of the
independent preprint of \cite{suzuki2026spectralgapunlabelledginibre}, made
public on August 18, 2026. Up to normalization of the Dirichlet form, its
finite-particle result is equivalent to Theorem \ref{th:main}, and its proof
uses essentially the same method. Our presentation nevertheless differs in
organization and emphasis and may be of independent interest. The first
version of this article was made public on August 19, 2026.

Still regarding computer-based automation, the Julia programming language was
used on ordinary laptop and desktop computers to conduct the numerical
experiments and create the graphics.

Finally, the author is grateful to Jeff Yao and the Tianyuan Mathematics
Research Center, as well as to Pierre Tarrès and the NYU-ECNU Institute of
Mathematical Sciences at NYU Shanghai, for providing remarkable working
conditions.

\section{Proof of Theorem \ref{th:main}}
\label{se:proof:th:main}

We use several steps. We first transfer the Ginibre Dirichlet form to a
Gaussian complex-analytic form by multiplication with the Vandermonde
polynomial. We then identify the relevant holomorphic projection and establish
the projection geometry behind the real half-distance estimate. Finally, we
verify sharpness on the center-of-mass observables and conclude the theorem.

\subsection{Vandermonde transform}

This step exploits the symmetry of the test functions, as well as the
holomorphy and antisymmetry of the Vandermonde determinant $V_n$. The
Vandermonde transform maps $f:\mathbb{C}^n\to\mathbb{C}$ to the product
$g=V_nf$, and we have
\begin{equation}\label{eq:vand-isometry}
  \|g\|_{L^2(\gamma_n)}^2=c_n'\|f\|_{L^2(\mu_n)}^2.
\end{equation}
It is convenient to introduce the normalized Vandermonde transform
$U:L^2(\mu_n)\to L^2(\gamma_n)$:
\begin{equation}\label{eq:U}
  Uf=\frac{1}{\sqrt{c_n'}}V_nf.
\end{equation}
Since $V_n\neq0$ almost everywhere, the linear map $U$ is unitary, it is in
particular an isometry: 
\begin{equation}
  \|Uf\|_{L^2(\gamma_n)}=\|f\|_{L^2(\mu_n)}
\end{equation}
for all $f\in L^2(\mu_n)$. In mathematical physics, this is known as a ground
state transform.

Since $V_n$ is a determinant, it is antisymmetric (or alternating) under
particle permutations:
\begin{equation}\label{eq:vand-alternating}
  V_n(z_{\sigma(1)},\ldots,z_{\sigma(n)})=\operatorname{sgn}(\sigma)V_n(z_1,\ldots,z_n).
\end{equation}
Also if $f$ is symmetric then $g$ is antisymmetric: the antisymmetry of the
Vandermonde polynomial $V_n$ is preserved when multiplied by a symmetric
function.

We introduce the partial complex derivatives with respect to
$z_k=x_k+\mathrm{i}y_k$ and $\bar z_k=x_k-\mathrm{i}y_k$:
\begin{equation}\label{eq:complex-derivatives}
  \partial_{z_k}=\frac12(\partial_{x_k}-\mathrm{i}\partial_{y_k}),
  \quad\text{and}\quad
  \bar\partial_{z_k}=\frac12(\partial_{x_k}+\mathrm{i}\partial_{y_k}).
\end{equation}

\begin{lemma}[Dirichlet form of real-valued functions]\label{le:grad}
  For all $n\geq1$ and $f\in\mathcal{C}^\infty_c(\mathbb{C}^n,\mathbb{R})$,
  \begin{equation}
    \int_{\mathbb{C}^n}|\nabla f|^2\mathrm{d}\mu_n
    =\frac{4}{c_n'}\sum_{k=1}^n\|\bar\partial_{z_k}g\|_{L^2(\gamma_n)}^2
    \quad\text{where}\quad
    g=fV_n.\label{eq:trans}
  \end{equation}
\end{lemma}

\begin{proof}
  Take $f\in\mathcal{C}^\infty_c(\mathbb{C}^n,\mathbb{R})$. Then
  $4|\partial_{z_k}f|^2=4|\bar\partial_{z_k}f|^2=|\partial_{x_k}f|^2+|\partial_{y_k}f|^2=|\nabla_{z_k}f|^2$,
  in particular
  \begin{equation}\label{eq:real-gradient}
    |\nabla f|^2=\sum_{k=1}^n|\nabla_{z_k}f|^2=4\sum_{k=1}^n |\bar\partial_{z_k}f|^2.
  \end{equation}
  Moreover, since $V_n$ is holomorphic, we have
  \begin{equation}\label{eq:dbar-vand}
    \bar\partial_{z_k}(V_n f)=V_n\bar\partial_{z_k}f.
  \end{equation}
  Putting all together, it follows that if
  $f\in\mathcal{C}^\infty_c(\mathbb{C}^n,\mathbb{R})$, then, with $g=V_nf$,
  \begin{align}
    \int_{\mathbb{C}^n}|\nabla f|^2\mathrm{d}\mu_n
    &=4\int_{\mathbb{C}^n}\sum_{k=1}^n|\bar\partial_{z_k} f|^2\mathrm{d}\mu_n\\
    &=\frac{4}{c_n'}\int_{\mathbb{C}^n}\sum_{k=1}^n|V_n\bar\partial_{z_k}f|^2\mathrm{d}\gamma_n\\
    &=\frac{4}{c_n'}\sum_{k=1}^n\int_{\mathbb{C}^n}|\bar\partial_{z_k}g|^2\mathrm{d}\gamma_n
      =\frac{4}{c_n'}\sum_{k=1}^n\|\bar\partial_{z_k}g\|_{L^2(\gamma_n)}^2.%\label{eq:trans}.
  \end{align}
\end{proof}

\subsection{Gaussian Hörmander--Berndtsson estimate}

This step exploits the strict plurisubharmonicity of the energy $n|z|^2$ of
the Gaussian Boltzmann--Gibbs measure $\gamma_n\propto\mathrm{e}^{-n|z|^2}$.

We denote by $\mathcal{C}^\infty_c(\mathbb{C}^n,\mathbb{C})$ the space of
functions $\mathbb{C}^n\to\mathbb{C}$ with real and imaginary parts in
$\mathcal{C}^\infty_c(\mathbb{C}^n,\mathbb{R})$. We recall the
differential-form-valued Dolbeault operators, given for 
$g\in\mathcal{C}^\infty(\mathbb{C}^n,\mathbb{C})$ by
\begin{equation}
  \partial g=\sum_{k=1}^n\partial_{z_k}g\mathrm{d}z_k
  \quad\text{and}\quad
  \bar\partial g=\sum_{k=1}^n\bar\partial_{z_k} g\mathrm{d}\bar{z}_k.
\end{equation}
See \cite{HormanderComplexAnalysisSeveralVariables},
\cite{BerndtssonThingsDbar}, and \cite{demailly}.

\begin{lemma}[Gaussian $\bar\partial$ Poincaré inequality or spectral
  gap]\label{le:dbar-gap}
  For all $n\geq1$ and $g\in\mathcal{C}^\infty_c(\mathbb{C}^n,\mathbb{C})$, 
  \begin{equation}%\label{eq:dbar-gap}
    \sum_{k=1}^n\|\bar\partial_{z_k}g\|_{L^2(\gamma_n)}^2
    \geq
    n\operatorname{dist}_{L^2(\gamma_n)}(g,\mathcal{H}_n)^2
  \end{equation}
  where $\mathcal{H}_n$ is the space of holomorphic functions on
  $\mathbb{C}^n$ that are in $L^2(\gamma_n)$.
\end{lemma}

\begin{proof}
  Let $H_{p,q}$ be the Hermite polynomials associated to $\gamma_n$, see
  Section \ref{se:hermite}. For all
  $g\in\mathcal{C}^\infty_c(\mathbb{C}^n,\mathbb{C})$, we write
  $g=\sum_{p,q}c_{p,q}H_{p,q}$ in $L^2(\gamma_n)$, which gives
  \begin{equation}
    \operatorname{proj}_{\mathcal{H}_n}(g)=\sum_p c_{p,0}H_{p,0},
    \quad\text{and hence}\quad
    \operatorname{dist}(g,\mathcal{H}_n)^2
    =\sum_{p}\sum_{q\neq0}|c_{p,q}|^2.
  \end{equation}
  % On the other hand, we have the ``lowering relation''
  % $\bar\partial_{z_k}H_{p,q}=\sqrt{nq_k}\,H_{p,q-e_k}$ where $e_k$ is the
  % $k$-th element of the canonical basis of $\mathbb{R}^n$, with convention
  % $H_{p,q-e_k}=0$ when $q_k=0$.
  By integration by parts and the lowering relation
  \eqref{eq:hermite:lowering}, we get
  $\langle\bar\partial_{z_k}g,H_{p,q}\rangle_{L^2(\gamma_n)}=\sqrt{n(q_k+1)}c_{p,q+e_k}$.
  Next, the Parseval identity gives, for all
  $g\in\mathcal{C}^\infty_c(\mathbb{C}^n,\mathbb{C})$,
  \begin{equation}
    \sum_{k=1}^n\|\bar\partial_{z_k}g\|_2^2
    =n\sum_{p,q}|q|\,|c_{p,q}|^2
    \ge n\sum_{p,q\neq0}|c_{p,q}|^2,
  \end{equation}
  which is the desired inequality. Equality holds if and only if $c_{p,q}=0$
  for all $|q|\ge2$.
\end{proof}

\begin{remark}[Class of functions and equality cases]
  Lemma \ref{le:dbar-gap} remains valid for all $g\in L^2(\gamma_n)$ with
  $\bar\partial_{z_k}g\in L^2(\gamma_n)$ for all $1\leq k\leq n$ where the
  derivatives are taken in a weak sense (Schwartz distributions). Indeed, a
  standard approximation argument (cutoff and mollification) shows that
  $\mathcal{C}^\infty_c(\mathbb{C}^n,\mathbb{C})$ is dense in the closed
  Gaussian $\bar\partial$-form domain for the graph norm. As an alternative to
  the Hermite approximation proof, there is the Hörmander--Berndtsson
  complex-analytic approach mentioned in Remark \ref{rm:horber}. The proof of
  Lemma \ref{le:dbar-gap} shows that equality is achieved if and only if
  \begin{equation}
    g=\sum_{p\in\mathbb{N}^n}a_pH_{p,0}+\sum_{k=1}^n\sum_{p\in\mathbb{N}^n}b_{k,p}H_{p,e_k}
  \end{equation}
  with
  $\sum_{p\in\mathbb{N}^n}|a_p|^2+\sum_{k=1}^n\sum_{p\in\mathbb{N}^n}|b_{k,p}|^2<\infty$.
  In other words, the equality space takes the form
  \begin{equation}
    \mathcal{H}_n\oplus\overline{\operatorname{span}}\{H_{p,e_k}:p\in\mathbb{N}^n,1\leq
    k\leq n\}.
  \end{equation}
  In particular, under the condition $g\perp\mathcal{H}_n$, equality holds if
  and only if
  \begin{equation}
    g=\sum_{k=1}^n\sum_{p\in\mathbb{N}^n}b_{k,p}H_{p,e_k}.
  \end{equation}
  With our normalization, we get $H_{p,0}(z)=\sqrt{\frac{n^{|p|}}{p!}}z^p$,
  which gives
  \begin{equation}
    H_{p,e_k}
    =\sqrt{n}\bar{z}_kH_{p,0}-\sqrt{p_k}H_{p-e_k,0}
    =(I-\operatorname{proj}_{\mathcal{H}_n})(\sqrt{n}\bar{z}_kH_{p,0})
  \end{equation}
  with $\sqrt{p_k}H_{p-e_k,0}=0$ when $p_k=0$. Finally, note that the zero
  function is the sole compactly supported test function satisfying the
  equality. Every equality case is real-analytic, hence a compactly supported
  equality case must vanish identically.
\end{remark}

\begin{remark}[Hörmander--Berndtsson inequality]\label{rm:horber}
  Lemma~\ref{le:dbar-gap} is the flat constant-curvature case of the weighted
  $L^2$ theory for the $\bar\partial$ equation developed by Lars Hörmander and
  Bo Berndtsson. More precisely, the density of $\gamma_n$ is proportional to
  $\mathrm{e}^{-\Phi(z)}$ where $\Phi(z):=n\sum_{k=1}^n|z_k|^2$, and
  \begin{equation}
    % {(\Phi_{j\bar k})}_{1\leq j,k\leq n}=
    {(\partial_{z_j}\bar\partial_{z_k}\Phi)}_{1\leq j,k\leq n}=nI_n.
  \end{equation}
  Since $\Phi-n\left|\cdot\right|^2$ is plurisubharmonic, it follows from
  \citet[Th.~1.1.1]{BerndtssonDbarMethods} that for all $\bar\partial$-closed
  $(0,1)$-form $\alpha=\sum_{k=1}^n\alpha_k\mathrm{d}{\bar z_k}$, in the sense
  that $\bar\partial\alpha=0$, with
  $ \sum_{k=1}^n\|\alpha_k\|_{L^2(\gamma_n)}^2<\infty$, there exists a
  solution $u\in L^2(\gamma_n)$ of the equation $\bar\partial u=\alpha$, in
  the sense of Schwartz distributions, satisfying %the quantitative estimate
  \begin{equation}\label{eq:flat-berndtsson}
    \|u\|_{L^2(\gamma_n)}^2
    \leq
    \int_{\mathbb{C}^n}\langle (\Phi_{j\bar k})^{-1}\alpha,\alpha\rangle_{\mathbb{C}^n}\mathrm{d}\gamma_n
    =\frac{1}{n}\sum_{k=1}^n\int_{\mathbb{C}^n}|\alpha_k|^2\mathrm{d}\gamma_n.
  \end{equation}
  This cannot hold for all solutions: if $h\in\mathcal{H}_n$, then $u+h$ is
  also a solution and can have arbitrarily large $L^2(\gamma_n)$ norm. To get
  Lemma~\ref{le:dbar-gap}, we take\footnote{It is closed since
    $\bar\partial\alpha=\bar\partial^2g=0$. The converse, stating that every
    $\bar\partial$-closed form is locally $\bar\partial$-exact is known as the
    $\bar\partial$-Poincaré or Dolbeault %--Grothendieck
    lemma, see \cite[Th.~2.3.3 and
    p.59]{HormanderComplexAnalysisSeveralVariables}. The Hörmander--Berndtsson
    theorem is stronger: under weighted $L^2$ hypotheses, it constructs a
    global solution and gives a quantitative norm estimate.}
  $\alpha=\bar\partial g$ and note that \eqref{eq:flat-berndtsson} holds for
  $u_*=g-\operatorname{proj}_{\mathcal{H}_n}(g)$. Indeed, from
  $\bar\partial(u-u_*)=0$ we get $u-u_*\in\mathcal{H}_n$, while
  $u_*=g-\operatorname{proj}_{\mathcal{H}_n}(g)\perp\mathcal{H}_n$, therefore
  $\|u\|_{L^2(\gamma_n)}^2=\|u_*\|_{L^2(\gamma_n)}^2+\|u-u_*\|_{L^2(\gamma_n)}^2\geq
  \|u_*\|_{L^2(\gamma_n)}^2$ (Pythagoras), hence \eqref{eq:flat-berndtsson}
  for $u_*$.
\end{remark}

\begin{remark}[Closedness of $\mathcal{H}_n$]
  The space $\mathcal{H}_n$ is known as the Segal--Bargmann space, Bargmann
  space, or Bargmann--Fock space associated to the weight $\gamma_n$. It is a
  closed subspace of $L^2(\gamma_n)$, see
  \cite[Prop.~14.15,~pp.~295--296]{HallQuantumTheory}. The proof uses a local
  pointwise estimate to show that an $L^2$-convergent sequence of entire
  functions converges locally uniformly, so its $L^2$ limit has an entire
  representative. It has nothing to do with probability or Gaussianity, and
  relies on the fact that the Radon--Nikodym derivative of the weight is
  bounded below by a positive constant on each compact.
\end{remark}

\subsection{Vandermonde holomorphic divisibility}

This step exploits symmetry of $f$ and of the density of $\gamma_n$, together
with holomorphy and antisymmetry for $V_n$.

\begin{lemma}[Divisibility]\label{le:divisibility}
  Let $\mathcal{H}_n^{\mathrm{asym}}$ be the subspace of $\mathcal{H}_n$ of
  antisymmetric elements of $\mathcal{H}_n$. Then
  \begin{equation}
    \mathcal{H}_n^{\mathrm{asym}}=V_n\mathcal{H}_n^{\mathrm{div}}
    \quad\text{where}\quad
    \mathcal{H}_n^{\mathrm{div}} = \{h\in
    L^2(\mu_n):V_nh\in\mathcal{H}_n^{\mathrm{asym}}\}.
  \end{equation}
  Equivalently, for every antisymmetric holomorphic $g\in L^2(\gamma_n)$, we have
  $g=V_nf$ for a symmetric holomorphic $f\in L^2(\mu_n)$. Moreover, for every
  symmetric $f\in L^2(\mu_n)$,
  \begin{equation}%\label{eq:distance-transfer}
    \operatorname{dist}_{L^2(\gamma_n)}(V_nf,\mathcal{H}_n)^2
    =
    c_n'\operatorname{dist}_{L^2(\mu_n)}(f,\mathcal{H}_n^{\mathrm{div}})^2.
  \end{equation}
\end{lemma}

It is immediate that $\mathcal{H}_n^{\mathrm{asym}}$ is a closed subspace of
$L^2(\gamma_n)$. Note that $\mathcal{H}_n^{\mathrm{div}}$ is a closed subspace
of $L^2(\mu_n)$ since multiplication by $V_n$ is an isometry from $L^2(\mu_n)$
to $L^2(\gamma_n)$, up to the factor $c_n'^{1/2}$.

We have $\mathcal{H}_n^{\mathrm{asym}}\supset V_n\mathcal{H}_n^{\mathrm{div}}$
by definition, and the lemma states that equality holds.

For a symmetric $f\in L^2(\mu_n)$, the product $V_nf$ is automatically
antisymmetric and in $L^2(\gamma_n)$, and membership in
$\mathcal{H}_n^{\mathrm{div}}$ additionally requires that this product be
holomorphic. The lemma states that this is equivalent to $f$ having an entire
symmetric representative. A counterexample is $\sum_{k=1}^n|z_k|^2$, which is
symmetric and in $L^2(\mu_n)$ but does not belong to
$\mathcal{H}_n^{\mathrm{div}}$ because it is not holomorphic.

\begin{proof}
  For $\sigma\in\mathfrak{S}_n$ and $G\in L^2(\gamma_n)$, we set
  $U_\sigma G(z)=G(z_{\sigma(1)},\ldots,z_{\sigma(n)})$, $z\in\mathbb{C}^n$.
  Since the density of $\gamma_n$ is symmetric, the linear map $U_\sigma$ is a
  unitary transform of $L^2(\gamma_n)$, for all $\sigma\in\mathfrak{S}_n$.
  Moreover, since holomorphy is preserved by coordinate permutation, we also
  get $U_\sigma\mathcal{H}_n=\mathcal{H}_n$ for all $\sigma\in\mathfrak{S}_n$.
  This gives the commutation
  $U_\sigma\operatorname{proj}_{\mathcal{H}_n}=\operatorname{proj}_{\mathcal{H}_n}U_\sigma$
  for all $\sigma\in\mathfrak{S}_n$. It follows that the orthogonal projection
  of an antisymmetric element $G$ of $L^2(\gamma_n)$ on $\mathcal{H}_n$ is
  antisymmetric. In particular
  \begin{equation}
    \operatorname{dist}(G,\mathcal{H}_n)=\operatorname{dist}(G,\mathcal{H}_n^{\mathrm{asym}}).
  \end{equation}
  Let $G\in\mathcal{H}_n$ be antisymmetric. If the transposition $(jk)$
  exchanges $z_j$ and $z_k$, then $G\circ(jk)=-G$. Hence $G=0$ on the
  hyperplane $\ell_{jk}:=z_j-z_k=0$. In local coordinates with $\ell_{jk}$ as
  one coordinate, the one-variable Weierstrass divisibility theorem, see
  \cite[Sec.~6.1 and 6.2]{HormanderComplexAnalysisSeveralVariables}, gives
  divisibility by $\ell_{jk}$. Still following \cite[Sec.~6.1 and
  6.2]{HormanderComplexAnalysisSeveralVariables}, to divide by all factors at
  once, we argue locally. Fix a point $p\in\mathbb{C}^n$. In the local ring
  $\mathcal O_p$ of holomorphic germs, the relevant linear forms $\ell_{jk}$
  with $\ell_{jk}(p)=0$ are nonassociate prime elements. Since the germ of $G$
  lies in each principal ideal $(\ell_{jk})$, the Weierstrass division theorem
  or unique-factorization property of $\mathcal O_p$, implies that it lies in
  the ideal generated by their product. The remaining factors $\ell_{jk}$ with
  $\ell_{jk}(p)\ne0$ are units near $p$. Thus $G/V_n$, initially defined on
  $\mathbb{C}^n_{\neq}$, has a holomorphic extension near every point $p$.
  Hence $G=V_n h$ with $h$ entire. On $\mathbb{C}^n_{\neq}$,
  \[
    h(\sigma z)=\frac{G(\sigma z)}{V_n(\sigma z)}
    =\frac{\operatorname{sgn}(\sigma)G(z)}{\operatorname{sgn}(\sigma)V_n(z)}=h(z),
  \]
  and the identity extends across \(\mathbb{C}^n_{=}\), hence \(h\) is
  symmetric. The norm identity
  \[
    \|h\|_{L^2(\mu_n)}^2=\frac{1}{c_n'}\|V_nh\|_{L^2(\gamma_n)}^2
  \]
  gives $h\in L^2(\mu_n)$, and the preceding symmetry and holomorphy
  properties give $h\in\mathcal{H}_n^{\mathrm{div}}$. Conversely, if
  $h\in\mathcal{H}_n^{\mathrm{div}}$, then $V_nh\in\mathcal{H}_n$ and is
  antisymmetric. This proves the first statement in the lemma. The distance
  identity follows from this bijection and the Vandermonde isometry
  \eqref{eq:vand-isometry} for $f-h$.
\end{proof}

\subsection{Holomorphic--antiholomorphic projection geometry}

Set
\[
  \mathcal{H}_n^{\mathrm{div},0}
  =\mathcal{H}_n^{\mathrm{div}}\cap\operatorname{span}\{1\}^{\perp},
  \qquad
  \overline{\mathcal{H}_n^{\mathrm{div},0}}
  =\{\overline h:h\in\mathcal{H}_n^{\mathrm{div},0}\},
\]
and let
$\overline{\mathcal{H}_n^{\mathrm{div}}}
=\{\overline h:h\in\mathcal{H}_n^{\mathrm{div}}\}$.
On the closed subspace of symmetric elements of $L^2(\mu_n)$, denote by
$P$, $Q$, and $P_0$ the orthogonal projections onto
$\mathcal{H}_n^{\mathrm{div}}$,
$\overline{\mathcal{H}_n^{\mathrm{div}}}$, and the constants, respectively.

\begin{lemma}[Holomorphic--antiholomorphic projection geometry]
  \label{le:projection-geometry}
  The following properties hold.
  \begin{enumerate}
  \item
    $\mathcal{H}_n^{\mathrm{div},0}$ and
    $\overline{\mathcal{H}_n^{\mathrm{div},0}}$ are orthogonal.
  \item For every symmetric $f\in L^2(\mu_n)$,
    \begin{equation}\label{eq:two-projection-bound}
      \|Pf\|_{L^2(\mu_n)}^2+\|Qf\|_{L^2(\mu_n)}^2
      \leq
      \|f\|_{L^2(\mu_n)}^2+\|P_0f\|_{L^2(\mu_n)}^2.
    \end{equation}
  \item If $f\in L^2(\mu_n)$ is real-valued, symmetric, and centered, then
    $h=Pf$, its pointwise complex conjugate $\overline h$, and
    $r=f-h-\overline h$, are three pairwise orthogonal functions in $L^2(\mu_n)$, and
    \begin{align}
      \|f\|_{L^2(\mu_n)}^2
      &=2\|h\|_{L^2(\mu_n)}^2+\|r\|_{L^2(\mu_n)}^2,
      \label{eq:projection-geometry-decomposition}\\
      \operatorname{dist}_{L^2(\mu_n)}
      \bigl(f,\mathcal{H}_n^{\mathrm{div}}\bigr)^2
      &=\frac{1}{2}\|f\|_{L^2(\mu_n)}^2
        +\frac{1}{2}\|r\|_{L^2(\mu_n)}^2.
      \label{eq:exact-half-distance}
    \end{align}
    In particular, we have the half-distance estimate 
    \begin{equation}\label{eq:half-distance-paper}
      \operatorname{dist}_{L^2(\mu_n)}
      \bigl(f,\mathcal{H}_n^{\mathrm{div}}\bigr)^2
      \geq\frac{1}{2}\|f\|_{L^2(\mu_n)}^2.
    \end{equation}
  \end{enumerate}
\end{lemma}

\begin{proof}
  For $\theta\in\mathbb{R}$, let
  $T_\theta f(z)=f(\mathrm{e}^{\mathrm{i}\theta}z)$. Since $\mu_n$ is
  invariant under global rotations, ${(T_\theta)}_{\theta\in\mathbb{R}}$ is a
  group of unitary transformations of $L^2(\mu_n)$. Let
  $h\in\mathcal{H}_n^{\mathrm{div}}$, set $G=V_nh$, and write
  $G=\sum_{m\geq0}G_m$ for the homogeneous expansion of $G$ in
  $\mathcal{H}_n$. Each $G_m$ is antisymmetric. By the divisibility Lemma
  \ref{le:divisibility}, each nonzero $G_m$ is divisible by $V_n$. Since $V_n$
  is homogeneous of degree $d_n=n(n-1)/2$, this gives $G_m=V_nh_{m-d_n}$ for a
  symmetric homogeneous holomorphic polynomial $h_{m-d_n}$ of degree $m-d_n$,
  in particular, $G_m=0$ for $m<d_n$. The Vandermonde isometry yields
  \[
    h=\sum_{k\geq0}h_k
    \quad\text{in }L^2(\mu_n),
  \]
  where $h_k$ has rotation weight $k$. Thus $\mathcal{H}_n^{\mathrm{div},0}$
  is the orthogonal sum of the strictly positive rotation-weight spaces,
  whereas $\overline{\mathcal{H}_n^{\mathrm{div},0}}$ is the orthogonal sum of
  the strictly negative rotation-weight spaces. This proves the first
  assertion. Let $P_+$ and $P_-$ be the orthogonal projections onto these two
  centered subspaces. Then $P=P_0+P_+$ and $Q=P_0+P_-$. To get
  \eqref{eq:two-projection-bound}, we observe that the ranges of $P_0$, $P_+$,
  and $P_-$ are pairwise orthogonal, hence the Bessel inequality gives
  \begin{align*}
    \|Pf\|_2^2+\|Qf\|_2^2
    &=2\|P_0f\|_2^2+\|P_+f\|_2^2+\|P_-f\|_2^2\\
    &\leq2\|P_0f\|_2^2+\|f-P_0f\|_2^2
      =\|f\|_2^2+\|P_0f\|_2^2.
  \end{align*}

  Let $C$ be complex conjugation. It is an antiunitary involution and $Q=CPC$.
  If $f$ is real-valued and centered, then $Qf=C(Pf)=\overline h$ and
  $P_0f=0$. Hence $h$, $\overline h$, and $r=f-h-\overline h$ are the
  orthogonal components of $f$ in the ranges of $P_+$, $P_-$, and their common
  orthogonal complement. This proves
  \eqref{eq:projection-geometry-decomposition}. Finally,
  \eqref{eq:exact-half-distance} follows from
  \eqref{eq:projection-geometry-decomposition} since
  \[
    \operatorname{dist}_{L^2(\mu_n)}
    \bigl(f,\mathcal{H}_n^{\mathrm{div}}\bigr)^2
    =\|f-h\|_2^2
    =\|\overline h+r\|_2^2
    =\|h\|_2^2+\|r\|_2^2.
  \]
\end{proof}

\subsection{Proof of the Poincaré inequality}
\label{ss:proof:poincare}

For all $n\geq1$ and all symmetric
$f\in\mathcal{C}^\infty_c(\mathbb{C}^n,\mathbb{R})$ centered with respect to
$\mu_n$, the successive usage of Lemmas \ref{le:grad}, \ref{le:dbar-gap},
\ref{le:divisibility}, and \ref{le:projection-geometry}, gives, denoting
$g=V_nf$,
\begin{align}
  \int_{\mathbb{C}^n}|\nabla f|^2\mathrm{d}\mu_n
  &=\frac{4}{c_n'}\sum_{k=1}^n\int|\bar\partial_{z_k}g|^2\mathrm{d}\gamma_n\\
  &\geq\frac{4}{c_n'}n\operatorname{dist}_{L^2(\gamma_n)}(g,\mathcal{H}_n)^2\label{eq:dbar-gap:inline}\\
  &=4n\operatorname{dist}_{L^2(\mu_n)}(f,\mathcal{H}_n^{\mathrm{div}})^2 \\
  &\geq 2n\|f\|_{L^2(\mu_n)}^2\label{eq:half-distance:inline}.
\end{align}
This shows the desired Poincaré inequality for symmetric and centered elements of
$\mathcal{C}^\infty_c(\mathbb{C}^n,\mathbb{R})$.

Next, consider a symmetric $f\in H^1_{\mathbb{R}}(\mu_n)$ centered with
respect to $\mu_n$. We first take a sequence $(u_m)$ in
$\mathcal{C}^\infty_c(\mathbb{C}^n,\mathbb{R})$ such that $u_m\to f$ as
$m\to\infty$ in $H^1_{\mathbb{R}}(\mu_n)$. Consider its symmetrization
$v_m=\frac{1}{n!}\sum_{\sigma\in\mathfrak{S}_n}u_m(z_{\sigma(1)},\ldots,z_{\sigma(n)})$.
Then thanks to the symmetry of the density of $\mu_n$, the symmetrization
contracts the $L^2(\mu_n)$ norm and the $L^2(\mu_n)$ norm of the gradient,
hence $v_m\to f$ in $H^1_{\mathbb{R}}(\mu_n)$. Now take a fixed symmetric
$\psi\in\mathcal{C}^\infty_c(\mathbb{C}^n,\mathbb{R})$ of unit mean with
respect to $\mu_n$, and define
$w_m=v_m-(\int_{\mathbb{C}^n}v_m\mathrm{d}\mu_n)\psi$. Then
$w_m\in\mathcal{C}^\infty_c(\mathbb{C}^n,\mathbb{R})$, is symmetric, and
centered. Since $f$ is centered and $v_m\to f$ in $L^2(\mu_n)$, we get
$\int_{\mathbb{C}^n}v_m\mathrm{d}\mu_n\to0$, and thus $w_m\to f$ in
$H^1_{\mathbb{R}}(\mu_n)$. This gives the desired Poincaré inequality for
symmetric and centered elements of $H^1_{\mathbb{R}}(\mu_n)$.

Finally, for a general symmetric $f\in H^1_{\mathbb{R}}(\mu_n)$, it suffices
to apply the result to $f-\int_{\mathbb{C}^n}f\mathrm{d}\mu_n$, and to note
that the constants belong to $H^1_{\mathbb{R}}(\mu_n)$ by the usual cutoff
approximation.

%Finally, it can be shown by standard approximation arguments that any centered
%symmetric $f\in H^1_{\mathbb{R}}(\mu_n)$ can be approximated by a sequence of
%centered symmetric elements of
%$\mathcal{C}^\infty_c(\mathbb{C}^n,\mathbb{R})$. 

\subsection{Proof of sharpness}
\label{ss:proof:sharp}

The observables $\varphi_\Re$ and $\varphi_\Im$ are symmetric and in
$\mathcal{C}^\infty(\mathbb{C}^n,\mathbb{R})\cap H^1_{\mathbb{R}}(\mu_n)$. We
have $\nabla\varphi_\Re(z)=(1,0,\ldots,1,0)\in\mathbb{R}^{2n}$ and
$\nabla\varphi_\Im(z)=(0,1,\ldots,0,1)\in\mathbb{R}^{2n}$, for all
$z\in\mathbb{C}^n$, hence
\begin{equation}
  \int_{\mathbb{C}^n}|\nabla\varphi_\Re|^2\mathrm{d}\mu_n
  =\int_{\mathbb{C}^n}|\nabla\varphi_\Im|^2\mathrm{d}\mu_n
  =n.
\end{equation}
On the other hand, by the factorization of Theorem \ref{th:split}, the image
measure of $\mu_n$ by the map
$z\in\mathbb{C}^n\mapsto z_1+\cdots+z_n\in\mathbb{C}$ is
$\mathcal{N}_{\mathbb{C}}(0,1)=\mathcal{N}_{\mathbb{R}}(0,\frac{1}{2})+\mathrm{i}\mathcal{N}_{\mathbb{R}}(0,\frac{1}{2})\equiv\mathcal{N}_{\mathbb{R}}(0,\frac{1}{2})^{\otimes
  2}$, in particular
\begin{equation}
  \int\varphi_\Re\mathrm{d}\mu_n
  =\int\varphi_\Im\mathrm{d}\mu_n=0,\quad
  \int_{\mathbb{C}^n}\varphi_\Re\varphi_\Im\mathrm{d}\mu_n
  =0,
  \quad\text{and}\quad 
  \int_{\mathbb{C}^n}\varphi_\Re^2\mathrm{d}\mu_n
  = \int_{\mathbb{C}^n}\varphi_\Im^2\mathrm{d}\mu_n
  =\frac{1}{2}.
\end{equation}
Therefore equality is achieved in \eqref{eq:pi} when $f$ is any linear
combination $a\varphi_\Re+b\varphi_\Im$, $a,b\in\mathbb{R}$.

\begin{remark}[Equality cases in the two estimates]
  Alternatively, or just for the sake of completeness, we can check that
  equality is achieved in \eqref{eq:dbar-gap:inline} and
  \eqref{eq:half-distance:inline}, when $f=\varphi_\Re$ or $f=\varphi_\Im$.
  Note that in this case, the test function $f$ is
  $\mathcal{C}^\infty(\mathbb{C}^n,\mathbb{R})$ but is not compactly
  supported, however the inequalities \eqref{eq:dbar-gap:inline} and
  \eqref{eq:half-distance:inline} remain valid by approximation. For all
  $z\in\mathbb{C}^n$ we have
  \begin{equation}\label{eq:real-M-split}
    \varphi_\Re(z)=\frac{1}{2}S(z)+\frac{1}{2}\bar{S}(z)
    \quad\text{where}\quad
    S(z)=z_1+\cdots+z_n
    \quad\text{and}\quad
    \bar{S}(z)=\bar{z}_1+\cdots+\bar{z}_n.
  \end{equation}
  The function $S$ is holomorphic and symmetric, hence
  $S\in\mathcal{H}_n^{\mathrm{div},0}$. Its conjugate $\bar{S}$ belongs to
  $\overline{\mathcal{H}_n^{\mathrm{div},0}}$, and these two subspaces are
  orthogonal by Lemma~\ref{le:projection-geometry}. Therefore,
  \begin{equation}\label{eq:projection-real-M}
    \operatorname{proj}_{\mathcal{H}_n^{\mathrm{div}}}(\varphi_\Re)=\frac{1}{2}S,
    \quad\text{and}\quad
    \operatorname{dist}_{L^2(\mu_n)}(\varphi_\Re,\mathcal{H}_n^{\mathrm{div}})^2
    =\frac{1}{4}\|\bar{S}\|^2_{L^2(\mu_n)}
    =\frac{1}{2}\|\varphi_\Re\|^2_{L^2(\mu_n)}.
  \end{equation}
  This proves equality in the half-distance estimate
  \eqref{eq:half-distance:inline}. It remains to check equality in the
  Hörmander--Berndtsson estimate \eqref{eq:dbar-gap:inline}. Let us define
  \begin{equation}\label{eq:g-real-M}
    g = V_n\varphi_\Re = \frac{1}{2}V_nS+\frac{1}{2}V_n\bar{S}.
  \end{equation}
  The first term of the right-hand side is holomorphic and antisymmetric. For
  the second term, we note that
  $\operatorname{proj}_{\mathcal{H}_n}(\bar z_kF)=\frac{1}{n}\partial_{z_k}F$
  for every holomorphic polynomial $F$ and all $1\leq k\leq n$, which gives
  \begin{equation}\label{eq:proj-bar-M-vand}
    \operatorname{proj}_{\mathcal{H}_n}(\bar{S}V_n) = \frac{1}{n}\sum_{k=1}^n\partial_{z_k}V_n.
  \end{equation}
  Recall that the Vandermonde polynomial $V_n(z)$ is translation-invariant in
  the sense that we have $V_n(z_1+c,\ldots,z_n+c) = V_n(z_1,\ldots,z_n)$ for
  all $c\in\mathbb{C}$ and $z\in\mathbb{C}^n$. Differentiating at $c=0$ gives
  \begin{equation}\label{eq:sum-partial-vand-zero}
    \sum_{k=1}^n\partial_{z_k}V_n(z)=0.
  \end{equation}
  It follows that $V_n(z)\bar{S}$ is orthogonal to $\mathcal{H}_n$. Moreover
  it is a first antiholomorphic Hermite mode: it is the degree-one
  antiholomorphic component
  \begin{equation}\label{eq:first-antiholomorphic-mode}
    \sum_{k=1}^n\Bigl(\bar{z}_kV_n-\frac{1}{n}\partial_{z_k}V_n\Bigr)
  \end{equation}
  with the correction term equal to zero by \eqref{eq:sum-partial-vand-zero}.
  Hence it lies in the eigenspace of
  $\sum_{k=1}^n\bar\partial_{z_k}^*\bar\partial_{z_k}$ with eigenvalue $n$. This gives
  equality in Lemma \ref{le:dbar-gap} for
  $g-\operatorname{proj}_{\mathcal{H}_n}g=\frac{1}{2}V_n\bar{S}$, thus equality in \eqref{eq:dbar-gap:inline}.
\end{remark}

\section{Alternative spectral proof of Theorem \ref{th:main}}
\label{se:proof:th:main:spectral}

We give here an alternative proof of the Poincaré inequality of Theorem
\ref{th:main} using the complex expression \eqref{eq:An:complex} of $A_n$ and
a spectral argument. It shares some of the aspects of the previous proof. A
priori, it remains Gaussian and it is not well suited for an extension to
plurisubharmonic potentials. All differential identities below are first
established on the symmetric core
$\mathcal{C}_c^\infty(\mathbb{C}^n_{\neq},\mathbb{C})$.

\begin{proof}

\noindent\emph{Step 1: factorization of the complex expression of $A_n$.}
We use a two-sided factorization by the Vandermonde polynomial.
More precisely, from \eqref{eq:An:complex} with $\alpha_n=n$, we get
\begin{equation}\label{eq:generator-two-sided-factorization}
  -\frac{n}{2}A_nf
  =
  \frac{1}{V_n}\mathcal{N}_{\bar z}(V_nf)
  +
  \frac{1}{\overline V_n}\mathcal{N}_{z}(\overline V_nf)
\end{equation}
where $\mathcal{N}_{\bar z}$ and $\mathcal{N}_{z}$ are special
Ornstein--Uhlenbeck operators, known as Gaussian ``number operators'':
\begin{equation}\label{eq:two-number-operators}
  \mathcal{N}_{\bar z}
  =\sum_{j=1}^n
  \left(-\partial_{z_j}+n\overline z_j\right)\bar\partial_{z_j}
  \quad\text{and}\quad
  \mathcal{N}_{z}
  =\sum_{j=1}^n
  \left(-\bar\partial_{z_j}+nz_j\right)\partial_{z_j}.
\end{equation}
Indeed, since $\frac{\partial_{z_j}V_n}{V_n}=\sum_{k\neq j}\frac{1}{z_j-z_k}$
on $\mathbb{C}^n_{\neq}$, we get, for all
$f\in\mathcal{C}_c^\infty(\mathbb{C}^n_{\neq},\mathbb{C})$,
\begin{align}
  \frac{1}{V_n}\mathcal{N}_{\bar z}(V_nf)
  &=-\sum_{j=1}^n\partial_{z_j}\bar\partial_{z_j}f
  +n\sum_{j=1}^n\overline z_j\bar\partial_{z_j}f
  -\sum_{j<k}
  \frac{(\bar\partial_{z_j}-\bar\partial_{z_k})f}{z_j-z_k},
  \label{eq:first-vandermonde-factor}\\
  \frac{1}{\overline V_n}\mathcal{N}_{z}(\overline V_nf)
  &=-\sum_{j=1}^n\partial_{z_j}\bar\partial_{z_j}f
  +n\sum_{j=1}^nz_j\partial_{z_j}f
  -\sum_{j<k}
  \frac{(\partial_{z_j}-\partial_{z_k})f}
  {\overline z_j-\overline z_k}.
  \label{eq:second-vandermonde-factor}
\end{align}

\noindent\emph{Step 2: the two Gaussian number-operator gaps.}
Let $\mathcal{H}_n$ and $\overline{\mathcal{H}_n}$ be the closed subspaces of
$L^2(\gamma_n)$ formed by entire holomorphic and entire antiholomorphic
functions, respectively. Let $H_{p,q}$ be the Hermite polynomials associated
to $\gamma_n$, see Section \ref{se:hermite}. We have
\begin{equation}
  \mathcal{N}_{\bar z}H_{p,q} = n|q|H_{p,q}
  \quad\text{and}\quad
  \mathcal{N}_{z}H_{p,q} = n|p|H_{p,q}.
\end{equation}
Moreover,
\begin{equation}
  \mathcal{H}_n = \overline{\operatorname{span}} \{H_{p,0}:p\in\mathbb{N}^n\}
  \quad\text{and}\quad
  \overline{\mathcal{H}_n} = \overline{\operatorname{span}} \{H_{0,q}:q\in\mathbb{N}^n\}.
\end{equation}
Consequently, in the sense of quadratic forms,
\begin{equation}\label{eq:two-number-operator-gaps}
  \mathcal{N}_{\bar z} \geq n\bigl(I-\operatorname{proj}_{\mathcal{H}_n}\bigr)
  \quad\text{and}\quad
  \mathcal{N}_{z} \geq n\bigl(I-\operatorname{proj}_{\overline{\mathcal{H}_n}}\bigr).
\end{equation}

\medskip

\noindent\emph{Step 3: pullback by $V_n$ and $\overline V_n$.}
On the symmetric subspace of $L^2(\mu_n)$, define the normalized Vandermonde
and anti-Vandermonde transforms (we have $U_+=U$ as in \eqref{eq:U})
\begin{equation}
  U_+f = \frac{1}{\sqrt{c_n'}}V_nf
  \quad\text{and}\quad
  U_-f = \frac{1}{\sqrt{c_n'}}\overline V_nf.
\end{equation}
Since $\mathrm{d}\mu_n=\frac{1}{c_n'}|V_n|^2\,\mathrm{d}\gamma_n$, both $U_+$
and $U_-$ are unitary maps from the symmetric subspace of $L^2(\mu_n)$ onto
the antisymmetric subspace of $L^2(\gamma_n)$. The orthogonal projections onto
$\mathcal{H}_n$ and $\overline{\mathcal{H}_n}$ preserve antisymmetry, as in
the proof of Lemma \ref{le:divisibility}. Lemma \ref{le:divisibility} and
complex conjugation therefore give
\begin{equation}\label{eq:two-pulled-projections}
  P = U_+^{-1}\operatorname{proj}_{\mathcal{H}_n}U_+
  \quad\text{and}\quad
  Q = U_-^{-1}\operatorname{proj}_{\overline{\mathcal{H}_n}}U_-,
\end{equation}
where $P$ and $Q$ are the projections introduced before Lemma
\ref{le:projection-geometry}. Hence Lemma \ref{le:projection-geometry} gives
\eqref{eq:two-projection-bound}.

\medskip

\noindent\emph{Step 4: combination of the two factors.}
Let
$f\in\mathcal{C}_c^\infty(\mathbb{C}^n_{\neq},\mathbb{C})$
be symmetric. Taking the scalar product of
\eqref{eq:generator-two-sided-factorization} with $f$, and using the
unitarity of $U_+$ and $U_-$, gives
\begin{equation}
  \mathcal{E}_{\mu_n}(f,f)
  =-\langle A_nf,f\rangle_{L^2(\mu_n)}
  =\frac{2}{n}\left(
    \langle U_+f,\mathcal{N}_{\bar z}U_+f\rangle_{L^2(\gamma_n)}
    +
    \langle U_-f,\mathcal{N}_{z}U_-f\rangle_{L^2(\gamma_n)}
  \right).
  \label{eq:form-from-two-factors}
\end{equation}
By \eqref{eq:two-number-operator-gaps},
\eqref{eq:two-pulled-projections}, and
\eqref{eq:two-projection-bound},
\begin{equation}
  \mathcal{E}_{\mu_n}(f,f)
  \geq
  2\left(
    2\|f\|_{L^2(\mu_n)}^2
    -\|Pf\|_{L^2(\mu_n)}^2
    -\|Qf\|_{L^2(\mu_n)}^2
  \right)
  \geq
  2\left(
    \|f\|_{L^2(\mu_n)}^2
    -\|P_0f\|_{L^2(\mu_n)}^2
  \right).
  \label{eq:form-gap-from-two-factors}
\end{equation}
%The symmetric part of
%$\mathcal{C}_c^\infty(\mathbb{C}^n_{\neq},\mathbb{C})$ is a form core for
%the restriction of $-L_n$ to the symmetric subspace: this follows by
%symmetrizing a sequence in the full form core. Hence
Finally \eqref{eq:form-gap-from-two-factors} extends to every symmetric
element $f$ of $H^1(\mu_n)$ by approximation.

\medskip

\noindent\emph{Step 5: passage to the spectral gap.}
The preceding form inequality says that, on the symmetric subspace, in the
sense of quadratic forms,
\begin{equation}\label{eq:symmetric-operator-lower-bound}
  -L_n\geq2(I-P_0).
\end{equation}
Constants belong to the kernel of $-L_n$, and
\eqref{eq:symmetric-operator-lower-bound} shows that there is no other kernel
vector in the symmetric subspace. Therefore the spectral measure of $-L_n$
associated with any symmetric $f\in\mathcal{D}(L_n)$ is supported on
$\{0\}\cup[2,\infty)$, which is the desired spectral gap estimate.
\end{proof}

\section{Alternative polynomial proof of Theorem \ref{th:main}}
\label{se:proof:th:main:polynomials}

We give an alternative proof of the Poincaré inequality in Theorem
\ref{th:main}. The idea is to expand the antisymmetric function $fV_n$ using
Hermite--Slater polynomials. This approach is a priori purely Gaussian and not
adapted for an extension to plurisubharmonic potentials.

\begin{proof}
  We first prove the inequality for a real-valued symmetric function
  $f\in\mathcal{C}_c^\infty(\mathbb{C}^n,\mathbb{R})$ centered with respect to
  $\mu_n$: $\int_{\mathbb{C}^n}f\,\mathrm{d}\mu_n=0$. Recall that
  $\mathrm{d}\mu_n=\frac{1}{c_n'}|V_n|^2\mathrm{d}\gamma_n$ and
  $\mathrm{d}\gamma_n(z)=(\frac{n}{\pi})^n\mathrm{e}^{-n|z|^2}\mathrm{d}z$. We
  shall diagonalize the antiholomorphic part of the Dirichlet form in the
  antisymmetric Hermite basis.

  Let $h_{a,b}$ be the Hermite polynomials associated to
  $\mathcal{N}_{\mathbb{C}}(0,1/n)$ the one-dimensional complex marginal of
  $\gamma_n$. Fix once and for all an ordering of $\mathbb{N}^2$. If
  $\Lambda=((a_1,b_1),\ldots,(a_n,b_n))$ consists of $n$ distinct pairs,
  written in increasing order, we consider the Hermite--Slater polynomial
  \begin{equation}\label{eq:slater-hermite}
    \Psi_\Lambda(z_1,\ldots,z_n)
    =
    \frac{1}{\sqrt{n!}}
    \det\bigl[h_{a_r,b_r}(z_s)\bigr]_{1\leq r,s\leq n}.
  \end{equation}
  The exterior-product construction of an orthonormal basis shows that the
  family ${( \Psi_\Lambda )}_\Lambda$ is an orthonormal basis of the
  antisymmetric subspace of $L^2(\gamma_n)$. The lowering relation
  \eqref{eq:one-variable-hermite-lowering} implies that
  \begin{equation}\label{eq:slater-number-operator}
    \Bigl(\sum_{j=1}^n(-\partial_{z_j}+n\overline z_j)\bar\partial_{z_j}\Bigr)\Psi_\Lambda
    = nq(\Lambda)\Psi_\Lambda
    \quad\text{where}\quad
    q(\Lambda) = \sum_{r=1}^n b_r.
  \end{equation}
  Indeed, the operator on the left is the sum of the identical one-particle
  ``number operators'' acting on the $n$ variables, and the determinant in
  \eqref{eq:slater-hermite} is therefore an eigenfunction with eigenvalue
  equal to the sum of the one-particle eigenvalues.

  Let $U$ be the normalized Vandermonde transform defined as in \eqref{eq:U}.
  Now, since $g=Uf$ is antisymmetric, it has an expansion
  \begin{equation}\label{eq:slater-expansion}
    g=\sum_\Lambda c_\Lambda\Psi_\Lambda
  \end{equation}
  in $L^2(\gamma_n)$. We have
  $|\nabla f|^2=4\sum_{j=1}^n|\bar\partial_{z_j}f|^2$ because $f$ is
  real-valued, and $c_n'^{-1/2}V_n\bar\partial_{z_j}f=\bar\partial_{z_j}g$
  since $V_n$ is holomorphic. Consequently, by the Parseval identity,
  integration by parts, and \eqref{eq:slater-number-operator}, we get
  \begin{equation}
    \frac{1}{n}\int_{\mathbb{C}^n}|\nabla f|^2\,\mathrm{d}\mu_n
    =\frac{4}{n}\sum_{j=1}^n\|\bar\partial_{z_j}g\|_{L^2(\gamma_n)}^2
    =4\sum_\Lambda q(\Lambda)|c_\Lambda|^2.
    \label{eq:slater-energy}
  \end{equation}
  We now examine the case $q(\Lambda)=0$. In that case all the integers
  $b_1,\ldots,b_n$ vanish. After writing $0\leq a_1<\cdots<a_n$, equations
  \eqref{eq:slater-hermite} and \eqref{eq:holomorphic-hermite} show that
  for a constant $C_\Lambda\neq0$,
  \begin{equation}\label{eq:holomorphic-slater-alternant}
    \Psi_\Lambda(z) = C_\Lambda \det\bigl[z_s^{a_r}\bigr]_{1\leq r,s\leq n}.
  \end{equation}
  Every alternating polynomial is divisible by the Vandermonde polynomial
  $V_n$. Hence the ratio $\Psi_\Lambda/V_n$ is a symmetric holomorphic
  polynomial. It is homogeneous of degree
  \begin{equation}\label{eq:alternant-degree}
    m(\Lambda)
    =
    \sum_{r=1}^n a_r-\frac{n(n-1)}{2}.
  \end{equation}
  Since the $a_r$ are distinct nonnegative integers,
  \[
    \sum_{r=1}^na_r \geq 0+1+\cdots+(n-1) = \frac{n(n-1)}{2}.
  \]
  Equality holds if and only if $(a_1,\ldots,a_n)=(0,1,\ldots,n-1)$, and in
  this exceptional case,
  \[
    \Psi_\Lambda = \pm \frac{1}{\sqrt{c_n'}}V_n = \pm U1,
  \]
  because both sides are normalized and $\det[z_s^{r-1}]_{1\leq r,s\leq n}=V_n(z)$.
  
  Let $g_0 = \sum_{\Lambda:q(\Lambda)=0} c_\Lambda\Psi_\Lambda$ be the
  orthogonal projection of $g$ onto $\mathcal{H}_n$. It is antisymmetric.
  Lemma \ref{le:divisibility} therefore gives a unique
  $h\in\mathcal{H}_n^{\mathrm{div}}$ such that
  \[
    g_0=Uh.
  \]
  Since $U$ is unitary and $g_0=\operatorname{proj}_{\mathcal{H}_n}g$, we have
  $h=Pf$. The alternant discussion above gives an explicit polynomial
  description of the homogeneous components of $h$. Since $f$ is real-valued
  and centered, Lemma \ref{le:projection-geometry} gives
  \begin{equation}\label{eq:holomorphic-projection-half-bound}
    2\|h\|_{L^2(\mu_n)}^2
    \leq
    \|f\|_{L^2(\mu_n)}^2.
  \end{equation}

  We can now conclude. Since $q(\Lambda)$ is a nonnegative integer, we have
  $q(\Lambda)\geq1$ whenever $q(\Lambda)\neq0$. It follows then from
  \eqref{eq:slater-energy} that
  \begin{align}
    \frac{1}{n}
    \int_{\mathbb{C}^n}|\nabla f|^2\,\mathrm{d}\mu_n
    &=
    4\sum_\Lambda q(\Lambda)|c_\Lambda|^2
    \notag\\
    &\geq
    4\sum_{\Lambda:q(\Lambda)\geq1}|c_\Lambda|^2
    \notag\\
    &=
    4\left(
      \|g\|_{L^2(\gamma_n)}^2
      -
      \|g_0\|_{L^2(\gamma_n)}^2
    \right)
    \notag\\
    &=
    4\left(
      \|f\|_{L^2(\mu_n)}^2
      -
      \|h\|_{L^2(\mu_n)}^2
    \right)
    \notag\\
    &\geq
    2\|f\|_{L^2(\mu_n)}^2.\nonumber
      %\label{eq:hermite-slater-poincare}
  \end{align}
  %Since $f$ is centered, this is the Poincaré inequality of Theorem
  %\ref{th:main}.
  By the symmetrization, centering, and density argument used in the first
  proof of Theorem \ref{th:main} given in Section \ref{se:proof:th:main}, the
  inequality extends to every symmetric $f\in H^1_{\mathbb{R}}(\mu_n)$.
\end{proof}

\section{Proof of Theorem \ref{th:poly}}
\label{se:proof:th:poly}

\begin{proof}
  Let $F$ be a polynomial in three variables and set
  \begin{equation}
    f(z)=F\bigl(S(z),\overline{S(z)},R(z)\bigr).
  \end{equation}
  We write $F_s,F_{\overline s},F_x$ for the derivatives of $F$ evaluated at
  $(S(z),\overline{S(z)},R(z))$. Let us compute now $A_nf$ using
  \eqref{eq:An:complex}. The elementary identities
  $\partial_{z_j}R=n\overline z_j-\overline S$ and
  $\bar\partial_{z_j}R=nz_j-S$ give
  \begin{equation}
    \sum_{j=1}^n\partial_{z_j}\bar\partial_{z_j}f
    =nF_{s\overline s}+nR F_{xx}+n(n-1)F_x
    \quad\text{and}\quad
    \sum_{j=1}^n
    \left(z_j\partial_{z_j}+\overline z_j\bar\partial_{z_j}\right)f
    =SF_s+\overline S F_{\overline s}+2R F_x.
  \end{equation}
  Finally, for every $j<k$,
  \begin{equation}
    (\partial_{z_j}-\partial_{z_k})f
    =n(\overline z_j-\overline z_k)F_x
    \quad\text{and}\quad
    (\bar\partial_{z_j}-\bar\partial_{z_k})f
    =n(z_j-z_k)F_x.
  \end{equation}
  Hence the singular terms in $A_nf$ cancel and
  \begin{equation}\label{eq:interaction-computation}
    \sum_{1\leq j<k\leq n}
    \left(
      \frac{(\partial_{z_j}-\partial_{z_k})f}
      {\overline z_j-\overline z_k}
      +
      \frac{(\bar\partial_{z_j}-\bar\partial_{z_k})f}
      {z_j-z_k}
    \right)
    =n^2(n-1)F_x.
  \end{equation}
  Injecting these formulas in \eqref{eq:An:complex} yields
  \begin{equation}\label{eq:restricted-generator}
    A_n f = \frac{\alpha_n}{n}\Bigl(
    4F_{s\overline s} -2\bigl(SF_s +\overline S F_{\overline s}\bigr) +4R F_{xx}+4(\kappa-R)F_x\Bigr).
  \end{equation}
  It follows that the restriction of $A_n$ to $\mathbb C[S,\overline S,R]$ is
  the sum of independent (complex) Hermite and (generalized) Laguerre
  operators. Now, as in the proof of Lemma \ref{le:dbar-gap}, see also
  \cite{Ito}, we have
  \begin{equation}\label{eq:hermite-equation}
    \left[
      4\partial_s\bar\partial_s
      -2\left(s\partial_s+\overline s\bar\partial_s\right)
    \right]\mathrm H_{a,b}
    =-2(a+b)\mathrm H_{a,b}.
  \end{equation}
  On the other hand, the Laguerre equation is
  \begin{equation}\label{eq:laguerre}
    \left[
      x\frac{\mathrm d^2}{\mathrm dx^2}
      +(\kappa-x)\frac{\mathrm d}{\mathrm dx}
    \right]L_m^{(\kappa-1)}(x)
    =-mL_m^{(\kappa-1)}(x),
  \end{equation}
  see \cite{Szego}. Applying \eqref{eq:restricted-generator} to the product in
  \eqref{eq:eigenpolynomial}, and then using \eqref{eq:hermite-equation} and
  \eqref{eq:laguerre}, gives \eqref{eq:eigenvalue}. Finally, the complex
  Hermite polynomials form a triangular basis of $\mathbb C[s,\overline s]$,
  and the generalized Laguerre polynomials form a triangular basis of
  $\mathbb C[x]$. Their products therefore form a basis of
  $\mathbb C[S,\overline S,R]$. The invariance of this algebra follows from
  \eqref{eq:eigenvalue}.

  In other words, we exhibit an algebra of symmetric polynomials preserved by
  $A_n$. On this algebra, $A_n$ reduces to the sum of a complex
  Ornstein--Uhlenbeck operator and a Laguerre operator. Consequently, products
  of complex Hermite and Laguerre polynomials form an explicit polynomial
  eigenbasis. But instead of using the differential operator $A_n$ and the
  algebra of differential calculus, we could get the result probabilistically
  as a corollary of the dynamical factorization of Theorem
  \ref{th:split:dynamical}.

  Finally, the orthogonality in $L^2(\mu_n)$ comes from the splitting of
  Theorem \ref{th:split} by factorization of the scalar products and the
  Hermite and Laguerre orthogonality relations.
\end{proof}

\section{Proof of Theorems \ref{th:deficit1} and \ref{th:deficit2}}
\label{se:proof:th:deficit}

\begin{proof}
  Take a symmetric $f\in\mathcal{C}^\infty_c(\mathbb{C}^n_{\neq},\mathbb{R})$.
  Recall that here $\alpha_n=n$, so that
  \begin{equation}\label{eq:EfA}
    \mathcal{E}_{\mu_n}(f,f)=-\langle
    A_nf,f\rangle_{L^2(\mu_n)}=\frac{1}{n}\int_{\mathbb{C}^n}|\nabla
    f|^2\,\mathrm{d}\mu_n,\quad\text{and}\quad
    \mathcal{E}_{\mu_n}(\widetilde f,\widetilde f)=\mathcal{E}_{\mu_n}(f,f).
  \end{equation}
  Let $U$ be the normalized Vandermonde transform as in \eqref{eq:U}. Recall
  that $g=U\widetilde f$ satisfies
  $\|g\|_{L^2(\gamma_n)}=\|\widetilde f\|_{L^2(\mu_n)}$, and is antisymmetric
  since $\widetilde f$ is symmetric. Now we use the expansion with respect to
  the Hermite polynomials $H_{p,q}$ associated to $\gamma_n$, see Section
  \ref{se:hermite}. For every $m\in\mathbb{N}$, let $g_m$ be the orthogonal
  projection of $g$ onto the closed span of the polynomials $H_{p,q}$ such
  that $|q|=m$. Thus
  \begin{equation}\label{eq:gamma-two-hermite-decomposition}
    g=\sum_{m=0}^{\infty}g_m
    \quad\text{in }L^2(\gamma_n).
  \end{equation}
  The subspace corresponding to a fixed value of $m$ is invariant under
  coordinate permutations. Its orthogonal projection therefore commutes with
  the permutation action, and each $g_m$ is antisymmetric. The lowering
  relation \eqref{eq:hermite:lowering} and the Parseval identity give
  \begin{equation}\label{eq:gamma-two-hermite-energy}
    \mathcal{E}_{\mu_n}(\widetilde f,\widetilde f)
    =4\sum_{m=1}^{\infty}m\|g_m\|_{L^2(\gamma_n)}^2.
  \end{equation}
  Indeed, since $\widetilde{f}$ is real valued and since $V_n$ is holomorphic,
  \begin{equation}
    \mathcal{E}_{\mu_n}(\widetilde f,\widetilde f)
    =\frac{4}{n}\sum_{j=1}^n\|\bar\partial_{z_j}\widetilde f\|_{L^2(\mu_n)}^2
    =\frac{4}{n}\sum_{j=1}^n\|\bar\partial_{z_j}g\|_{L^2(\gamma_n)}^2
    =\frac{4}{n}n\sum_{m\geq1}m\|g_m\|_{L^2(\gamma_n)}^2.
  \end{equation}
  The component $g_0=\operatorname{proj}_{\mathcal{H}_n}g$ belongs to
  $\mathcal{H}_n$ and is antisymmetric. By the divisibility Lemma
  \ref{le:divisibility}, there exists a unique
  $h\in\mathcal{H}_n^{\mathrm{div}}$ such that $g_0=Uh$, and since $U$ is
  unitary, $h=P\widetilde f$. Define $r=\widetilde f-h-\overline h$. The
  function $\widetilde f$ is real-valued and centered. The projection geometry
  Lemma \ref{le:projection-geometry} therefore shows that $h$, $\overline h$,
  and $r$ are pairwise orthogonal and gives
  \eqref{eq:projection-geometry-decomposition}. On the other hand, 
  \begin{equation}\label{eq:gamma-two-zero-level-norm}
    \|g_0\|_{L^2(\gamma_n)}=\|h\|_{L^2(\mu_n)}.
  \end{equation}
  Now, from \eqref{eq:gamma-two-hermite-energy} and
  \eqref{eq:gamma-two-zero-level-norm}, we get
  \begin{align*}
    \mathcal{E}_{\mu_n}(\widetilde f,\widetilde f)
    &=4\sum_{m=1}^{\infty}\|g_m\|_{L^2(\gamma_n)}^2
      +4\sum_{m=2}^{\infty}(m-1)\|g_m\|_{L^2(\gamma_n)}^2\\
    &=4\left(
      \|\widetilde f\|_{L^2(\mu_n)}^2
      -\|h\|_{L^2(\mu_n)}^2
      \right)
      +4\sum_{m=2}^{\infty}(m-1)\|g_m\|_{L^2(\gamma_n)}^2.
  \end{align*}
  Subtracting $2\|\widetilde f\|_{L^2(\mu_n)}^2$ and using
  \eqref{eq:projection-geometry-decomposition} yields the desired sum-of-squares
  formula for the Poincaré deficit \eqref{eq:poincare-deficit-sum-of-squares}.
  Finally, since
  $\langle A_n\widetilde f,\widetilde
  f\rangle_{L^2(\mu_n)}=-\mathcal{E}_{\mu_n}(\widetilde f,\widetilde f)$, we
  get
  \begin{equation}\label{eq:gamma-two-algebraic-identity}
    \|A_nf\|_{L^2(\mu_n)}^2
    -2\mathcal{E}_{\mu_n}(f,f)
    =
    \|(A_n+2)\widetilde f\|_{L^2(\mu_n)}^2
    +2\Bigl(
      \mathcal{E}_{\mu_n}(\widetilde f,\widetilde f)
      -2\|\widetilde f\|_{L^2(\mu_n)}^2
    \Bigr).
  \end{equation}
  Combining \eqref{eq:poincare-deficit-sum-of-squares} and
  \eqref{eq:gamma-two-algebraic-identity}, we obtain the sum-of-squares
  formula for the $\Gamma_2$ deficit
  \eqref{eq:integrated-gamma-two-sum-of-squares}.
\end{proof}

\begin{proof}
  Recall that we take $\alpha_n=n$, giving \eqref{eq:EfA}. We first introduce
  the Gaussian operator that will replace the Hermite expansion. On
  $L^2(\gamma_n)$, consider the nonnegative functional quadratic form
  \begin{equation}\label{eq:gaussian-dbar-form}
    Q_n(f,f) = \sum_{j=1}^n \|\bar\partial_{z_j}f\|_{L^2(\gamma_n)}^2.
  \end{equation}
  %with domain consisting of the functions $w\in L^2(\gamma_n)$ whose weak
  %derivatives $\bar\partial_{z_j}w$ belong to $L^2(\gamma_n)$ for every
  %$j$.
  The associated nonnegative self-adjoint operator is
  $\mathcal{N}_n=\sum_{j=1}^n\mathcal{N}_{\bar{z}_j}=\sum_{j=1}^n(-\partial_{z_j}+n\overline
  z_j)\bar\partial_{z_j}$. Now we record the Gaussian Bochner--Kodaira
  identity. By Gaussian integration by parts, we get
  \[
    \mathcal{N}_n=\sum_{j=1}^nD_j^*D_j
    \qquad\text{and}\qquad
    [D_j,D_k^*]=n\delta_{jk}
    \quad\text{where}\quad
    D_j=\bar\partial_{z_j}
    \quad\text{and}\quad
    D_j^*=-\partial_{z_j}+n\overline z_j.
  \]
  Hence, for
  $w\in\mathcal{C}_c^\infty(\mathbb{C}^n,\mathbb{C})$,
  \begin{align}
    \|\mathcal{N}_nw\|_{L^2(\gamma_n)}^2
    &=
    \sum_{j,k=1}^n
    \langle D_j^*D_jw,D_k^*D_kw\rangle_{L^2(\gamma_n)}
    \notag\\
    &=
    \sum_{j,k=1}^n
    \langle D_jw,D_jD_k^*D_kw\rangle_{L^2(\gamma_n)}
    \notag\\
    &=
    \sum_{j,k=1}^n
    \langle D_jw,D_k^*D_jD_kw\rangle_{L^2(\gamma_n)}
    +
    n\sum_{j=1}^n
    \|D_jw\|_{L^2(\gamma_n)}^2
    \notag\\
    &=
    \sum_{j,k=1}^n
    \|D_jD_kw\|_{L^2(\gamma_n)}^2
    +
    n\sum_{j=1}^n
    \|D_jw\|_{L^2(\gamma_n)}^2.
  \end{align}
  Therefore
  \begin{equation}\label{eq:gaussian-bochner-kodaira}
    \|\mathcal{N}_nw\|_{L^2(\gamma_n)}^2
    -
    n\langle w,\mathcal{N}_nw\rangle_{L^2(\gamma_n)}
    =
    \sum_{j,k=1}^n
    \|\bar\partial_{z_j}\bar\partial_{z_k}w\|_{L^2(\gamma_n)}^2.
  \end{equation}
  The space $\mathcal{C}_c^\infty(\mathbb{C}^n,\mathbb{C})$ is an operator
  core for $\mathcal{N}_n$, as follows by radial cutoff and mollification.
  Consequently, the identity \eqref{eq:gaussian-bochner-kodaira} extends by
  closure to every $w\in\mathcal{D}(\mathcal{N}_n)$. Next, the kernel of
  $\mathcal{N}_n$ is the set $\mathcal{H}_n$ of entire holomorphic functions
  in $L^2(\gamma_n)$. Indeed,
  \[
    \langle w,\mathcal{N}_nw\rangle_{L^2(\gamma_n)}
    =
    \sum_{j=1}^n
    \|\bar\partial_{z_j}w\|_{L^2(\gamma_n)}^2,
  \]
  so a function belongs to the kernel if and only if all its $\bar\partial$
  derivatives vanish weakly. Such a function has an entire holomorphic
  representative. The identity \eqref{eq:gaussian-bochner-kodaira} implies
  $\mathcal{N}_n^2\geq n\mathcal{N}_n$ in the sense of self-adjoint operators.
  Since $\mathcal{N}_n$ is nonnegative, the spectral theorem gives
  $\operatorname{spec}(\mathcal{N}_n) \subset \{0\}\cup[n,\infty)$. In
  particular, $\mathcal{N}_n\geq n$ on $\mathcal{H}_n^\perp$. Therefore
  $\mathcal{N}_n$ is invertible on $\mathcal{H}_n^\perp$ and bounded below by
  $nI$.

  Let $\Pi_n=\operatorname{proj}_{\mathcal{H}_n}$ be the orthogonal projection
  of $L^2(\gamma_n)$ onto $\mathcal{H}_n$. We now transfer to $\mu_n$. Let $U$
  be the normalized Vandermonde transform as in \eqref{eq:U}. Recall that
  $g=U\widetilde f$ satisfies
  $\|g\|_{L^2(\gamma_n)}=\|\widetilde f\|_{L^2(\mu_n)}$, and is antisymmetric
  since $\widetilde f$ is symmetric. Since $V_n$ is holomorphic, we
  get $\bar\partial_{z_j}g=U\bar\partial_{z_j}\widetilde f$. It follows first
  for smooth functions, and then by density in the form domain, that
  \begin{equation}\label{eq:gamma-two-energy-number-operator}
    \mathcal{E}_{\mu_n}(\widetilde f,\widetilde f) = \frac{4}{n}Q_n(g,g).
  \end{equation}
  Define $g_0=\Pi_ng$ and $u=(I-\Pi_n)g$. The projection $\Pi_n$ commutes with
  coordinate permutations, because both $\gamma_n$ and $\mathcal{N}_n$ are
  invariant under such permutations. Hence $g_0$ remains antisymmetric.

  By the divisibility Lemma \ref{le:divisibility}, there exists a unique
  $h\in\mathcal{H}_n^{\mathrm{div}}$ such that $g_0=Uh$. Since $U$ is unitary
  and $g_0=\Pi_ng$, we have $h=P\widetilde f$. Set
  $r=\widetilde f-h-\overline h$. The projection geometry Lemma
  \ref{le:projection-geometry} gives
  \eqref{eq:projection-geometry-decomposition}. Moreover,
  we have $\|g_0\|_{L^2(\gamma_n)}=\|h\|_{L^2(\mu_n)}$, and hence
  \begin{equation}\label{eq:gamma-two-u-norm}
    \|u\|_{L^2(\gamma_n)}^2
    = \|\widetilde f\|_{L^2(\mu_n)}^2 - \|h\|_{L^2(\mu_n)}^2
    = \|h\|_{L^2(\mu_n)}^2 + \|r\|_{L^2(\mu_n)}^2.
  \end{equation}
  Since $g_0$ belongs to the kernel of $\mathcal{N}_n$, we have
  \begin{equation}
    Q_n(g,g)=Q_n(u,u).
  \end{equation}
  Moreover, $u\in\mathcal{H}_n^\perp$, so the spectral theorem gives
  \begin{equation}\label{eq:gamma-two-number-square}
    Q_n(u,u) = n\|u\|_{L^2(\gamma_n)}^2 +
    \|(\mathcal{N}_n-nI)^{1/2}u\|_{L^2(\gamma_n)}^2.
  \end{equation}
  Here $(\mathcal{N}_n-nI)^{1/2}$ is defined by spectral calculus on
  $\mathcal{H}_n^\perp$.

  Combining \eqref{eq:gamma-two-energy-number-operator},
  \eqref{eq:projection-geometry-decomposition}, \eqref{eq:gamma-two-u-norm}, and
  \eqref{eq:gamma-two-number-square}, we obtain
  \begin{equation}\label{eq:poincare-deficit-with-number-operator}
    \mathcal{E}_{\mu_n}(\widetilde f,\widetilde f)
    -2\|\widetilde f\|_{L^2(\mu_n)}^2
    =2\|r\|_{L^2(\mu_n)}^2
    +\frac{4}{n}\|(\mathcal{N}_n-nI)^{1/2}u\|_{L^2(\gamma_n)}^2.
  \end{equation}

  Formula \eqref{eq:poincare-deficit-with-number-operator} is already an exact
  sum of two squared norms. We next rewrite its second square as a sum of
  squared derivatives. Define $\mathcal{N}_n^{-1/2}$ by spectral calculus,
  with value zero on $\mathcal{H}_n$, and set
  \begin{equation}\label{eq:gamma-two-pseudoinverse}
    v=\mathcal{N}_n^{-1/2}u.
  \end{equation}
  Since $u\in\mathcal{H}_n^\perp$ and $u\in\mathcal{D}(\mathcal{N}_n^{1/2})$,
  the spectral gap of $\mathcal{N}_n$ implies that
  $v\in\mathcal{D}(\mathcal{N}_n)$ and $u=\mathcal{N}_n^{1/2}v$. Applying the
  closed version of the Bochner--Kodaira identity
  \eqref{eq:gaussian-bochner-kodaira} to $v$ gives
  \begin{equation} \label{eq:gamma-two-second-derivative-square}
    \|(\mathcal{N}_n-nI)^{1/2}u\|_{L^2(\gamma_n)}^2 %
    = \|\mathcal{N}_nv\|_{L^2(\gamma_n)}^2 %
    - n\langle v,\mathcal{N}_nv\rangle_{L^2(\gamma_n)} %
    = \sum_{j,k=1}^n\|\bar\partial_{z_j}\bar\partial_{z_k}v\|_{L^2(\gamma_n)}^2.
  \end{equation}
  Used in \eqref{eq:poincare-deficit-with-number-operator}, it gives the
  sum-of-squares formula \eqref{eq:poincare-deficit-differential-squares} for
  the Poincaré inequality. The derivation of the sum-of-squares formula
  \eqref{eq:gamma-two-deficit-differential-squares} for $\Gamma_2$ can be done
  as in the proof of Theorem \ref{th:deficit1}.
\end{proof}

\section{Proof of Theorem \ref{th:rlsi}}
\label{se:proof:th:rlsi}

\begin{proof}
  Following \cite{Kostlan}, for all measurable, integrable, and symmetric
  $G:(0,\infty)^n\to\mathbb{R}$,
  \begin{equation}\label{eq:kostlan-identity-radial-lsi}
    \int_{\mathbb{C}^n}G(|z_1|,\ldots,|z_n|)\,\mathrm{d}\mu_n(z)
    =
    \int_{(0,\infty)^n}G(r_1,\ldots,r_n)\prod_{k=1}^n\mathrm{d}\nu_{n,k}(r_k)
  \end{equation}
  where
  $\mathrm{d}\nu_{n,k}(r)=\frac{2n^k}{\Gamma(k)}r^{2k-1}\mathrm{e}^{-nr^2}\mathrm{d}r$,
  in other words $R_k\sim\nu_{n,k}$ if and only if
  $nR_k^2\sim\mathrm{Gamma}(k,1)$. Equivalently, the unordered multiset
  $\{|Z_1|,\ldots,|Z_n|\}$ under $\mu_n$ has the same distribution as the
  unordered multiset of independent random variables $R_1,\ldots,R_n$ with
  respective distributions $\nu_{n,1},\ldots,\nu_{n,n}$. Beyond the complex
  Ginibre log-gas, such a factorization property remains valid for
  determinantal point processes with radial kernel, see for instance
  \cite[Th.~26]{HoughKrishnapurPeresViragIndependence}.

  Let us now prove a log-Sobolev inequality for
  $\nu_{n,1}\otimes\cdots\otimes\nu_{n,n}$. For all $1\leq k\leq n$, we take
  $X_k\sim\mathcal{N}_{\mathbb{C}^k}(0,1/n)=(\frac{n}{\pi})^k\mathrm{e}^{-n|x|^2}\mathrm{d}x_1\cdots\mathrm{d}x_k$,
  a random vector of $\mathbb{C}^k\equiv\mathbb{R}^{2k}$ following the
  $k$-dimensional complex marginal of $\gamma_n$. We take $X_1,\ldots,X_n$
  independent. We have $|X_k|\sim\nu_{n,k}$. Define
  \[
    H(x_1,\ldots,x_n)
    =
    F(|x_1|,\ldots,|x_n|),
    \qquad
    x_k\in\mathbb{R}^{2k}.
  \]
  The joint distribution $\eta_n$ of $(X_1,\ldots,X_n)$ is the Gaussian
  distribution $\mathcal{N}_{\mathbb{R}^d}(0,(2n)^{-1}I_d)$, where $d=n(n+1)$.
  Now the optimal Gaussian logarithmic Sobolev inequality gives
  \begin{equation}\label{eq:gaussian-lsi-for-kostlan}
    \operatorname{Ent}_{\eta_n}(H^2)
    \leq
    \frac{1}{n}\int_{\mathbb{R}^d}|\nabla H|^2\mathrm{d}\eta_n,
  \end{equation}
  see \cite{Gross} or \cite[Ch.~5]{BakryGentilLedoux}. At every point for
  which $x_k\neq0$ for all $k$, the chain rule gives
  $\nabla_{x_k}H(x_1,\ldots,x_n)=\partial_kF(|x_1|,\ldots,|x_n|)\frac{x_k}{|x_k|}$.
  Consequently, we have, almost everywhere,
  \begin{equation}\label{eq:gaussian-block-gradient}
    |\nabla H(x_1,\ldots,x_n)|^2 = \sum_{k=1}^n|\partial_kF(|x_1|,\ldots,|x_n|)|^2.
  \end{equation}
  Using the image measure by the map
  $(x_1,\ldots,x_n)\mapsto(|x_1|,\ldots,|x_n|)$ in
  \eqref{eq:gaussian-lsi-for-kostlan} yields
  \begin{equation}\label{eq:kostlan-product-lsi}
    \operatorname{Ent}_{\nu_{n,1}\otimes\cdots\otimes\nu_{n,n}}(F^2)
    \leq
    \frac{1}{n}
    \int_{(0,\infty)^n}
    \sum_{k=1}^n|\partial_kF(r)|^2
    \prod_{k=1}^n\mathrm{d}\nu_{n,k}(r_k).
  \end{equation}

  We now transfer this inequality to $\mu_n$. Since $F$ is symmetric, the
  functions $F^2$ and $F^2\log F^2$ are also symmetric. The Kostlan identity
  \eqref{eq:kostlan-identity-radial-lsi} gives then
  \begin{equation}\label{eq:kostlan-entropy-transfer}
    \operatorname{Ent}_{\mu_n}(f^2)
    =
    \operatorname{Ent}_{\nu_{n,1}\otimes\cdots\otimes\nu_{n,n}}(F^2).
  \end{equation}
  If $z_j=x_j+\mathrm{i}y_j\neq0$, by Euclidean chain rule,
  $\nabla_{(x_j,y_j)}f(z)=\partial_jF(|z_1|,\ldots,|z_n|)\frac{(x_j,y_j)}{|z_j|}$,
  hence
  \begin{equation}\label{eq:radial-gradient-identity}
    |\nabla f(z)|^2
    =
    \sum_{j=1}^n
    \left|
      \partial_jF(|z_1|,\ldots,|z_n|)
    \right|^2
  \end{equation}
  for almost every $z\in\mathbb{C}^n$. Differentiating the symmetry identity
  of $F$ shows that $r\longmapsto\sum_{j=1}^n|\partial_jF(r)|^2$ is itself
  symmetric. Now a second application of the Kostlan identity
  \eqref{eq:kostlan-identity-radial-lsi} therefore gives
  \begin{align}
    \int_{\mathbb{C}^n}|\nabla f|^2\,\mathrm{d}\mu_n
    &=
    \int_{(0,\infty)^n}
    \sum_{j=1}^n|\partial_jF(r)|^2
    \prod_{k=1}^n\mathrm{d}\nu_{n,k}(r_k).
    \label{eq:kostlan-energy-transfer}
  \end{align}
  Combining \eqref{eq:kostlan-product-lsi},
  \eqref{eq:kostlan-entropy-transfer}, and
  \eqref{eq:kostlan-energy-transfer} proves
  the desired logarithmic Sobolev inequality \eqref{eq:rlsi}.

  Finally, let $(f_m)$ be a sequence of radial symmetric functions in
  $\mathcal{C}^\infty_c(\mathbb{C}^n,\mathbb{R})$ converging to $f$ in
  $H^1_{\mathbb{R}}(\mu_n)$. Then $f_m^2\to f^2$ in $L^1(\mu_n)$. The entropy
  functional is lower semicontinuous under this convergence, while the
  Dirichlet energies converge. Therefore
  \[
    \operatorname{Ent}_{\mu_n}(f^2)
    \leq
    \liminf_{m\to\infty}
    \operatorname{Ent}_{\mu_n}(f_m^2)
    \leq
    \lim_{m\to\infty}\mathcal{E}_{\mu_n}(f_m,f_m)
    =
    \mathcal{E}_{\mu_n}(f,f).
  \]
  This proves the extension to the radial symmetric Sobolev closure.
\end{proof}

\appendix

\section{Approximation lemmas}

\begin{lemma}[Domain of the closure of the gradient]\label{le:gradient:closure:domain}
  For all $n\geq1$, we can identify $H^1_{\mathbb{R}}(\mu_n)$ with the domain
  of the closure of
  $\nabla:\mathcal{C}^\infty_c(\mathbb{C}^n,\mathbb{R})\to
  L^2(\mu_n,\mathbb{R}^{2n})$ equipped with the graph norm
  \eqref{eq:graphnorm}.
\end{lemma}

\begin{proof}
  If ${(f_m)}$ is a sequence in
  $\mathcal{C}^\infty_c(\mathbb{C}^n,\mathbb{R})$ with $f_m\to 0$ in
  $L^2(\mu_n)$ and $\nabla f_m\to G$ in $L^2(\mu_n,\mathbb{R}^{2n})$, then on
  every relatively compact open subset $O$ of $\mathbb{C}^n_{\neq}$, the
  density of $\mu_n$ is $\mathcal{C}^\infty$ and is bounded above and below by
  positive constants, hence $f_m\to0$ and $\nabla f_m\to G$ in the
  corresponding $L^2$ spaces. By distributional integration by parts, we get
  $G=0$ on $O$. Since $O$ is arbitrary, we get that $G=0$ on
  $\mathbb{C}^n_{\neq}$, and we know that $\mu_n(\mathbb{C}^n_{\neq})=1$. This
  proves that $\nabla$ is closable.
\end{proof}

\begin{lemma}[Equivalence]\label{le:closure}
  For all $n\geq1$, the graph norms \eqref{eq:graphnorm} and
  \eqref{eq:graph:norm:En} are equivalent, and the spaces
  $\mathcal{C}^\infty_c(\mathbb{C}^n,\mathbb{C})$ and
  $\mathcal{C}^\infty_c(\mathbb{C}^n_{\neq},\mathbb{C})$ have the same closure
  for the graph norm, denoted $H^1(\mu_n)$.
\end{lemma}

\begin{proof}
  The equivalence of the norms comes from the fact that $\alpha_n/\beta_n>0$.
  Next, near the hyperplane $\{z\in\mathbb{C}^n:z_j=z_k\}$, set $u=z_j-z_k$.
  The density of $\mu_n$ contains a factor $|u|^2$. Let $\chi_\varepsilon$ be
  a smooth cutoff function vanishing for $|u|\leq\varepsilon$, equal to $1$
  for $|u|\geq2\varepsilon$, and such that
  $|\nabla\chi_\varepsilon|\leq C/\varepsilon$. On each compact $K$, we get
  $\int_K|1-\chi_\varepsilon|^2\mathrm{d}\mu_n=O(\varepsilon^4)$, whereas
  $\int_K|\nabla\chi_\varepsilon|^2\mathrm{d}\mu_n\leq
  C'\varepsilon^{-2}\int_0^{2\varepsilon}r^2r\mathrm{d}r=O(\varepsilon^2)$.
  Now define $\eta_\varepsilon(z)=\prod_{j<k}\chi_\varepsilon(z_j-z_k)$. We
  have
  $|\nabla\eta_\varepsilon|^2\leq
  C''\sum_{j<k}|\nabla\chi_\varepsilon(z_j-z_k)|^2$. For every
  $h\in\mathcal{C}^\infty_c(\mathbb{C}^n)$, we have
  $h\eta_\varepsilon\in\mathcal{C}^\infty_c(\mathbb{C}^n_{\neq})$, and
  $h\eta_\varepsilon\to h$ for the graph norm since
  $\nabla(h(1-\eta_\varepsilon))=(1-\eta_\varepsilon)\nabla
  h-h\nabla\eta_\varepsilon$.
\end{proof}

\begin{lemma}[Capacity]\label{le:capacity}
  For all $n\geq1$, the set $\mathbb{C}^n_{=}$ has zero capacity.
\end{lemma}

\begin{proof}
  If $n=1$, then the set is empty. Suppose that $n\geq2$. Since
  $1\in H^1(\mu_n)$, there exists a sequence ${(v_m)}$ in
  $\mathcal{C}^\infty_c(\mathbb{C}^n_{\neq},\mathbb{R})$ such that $v_m\to1$
  in $H^1(\mu_n)$. Now $1-v_m$ is equal to $1$ in a neighborhood of
  $\mathbb{C}^n_{=}$. Hence
  $\mathrm{Cap}_n(\mathbb{C}^n_{=})\leq\|1-v_m\|_{L^2(\mu_n)}^2+\mathcal{E}_n(1-v_m,1-v_m)\to0$.
  Here we define the capacity of $C\subset\mathbb{C}^n$ as follows:
  $\mathrm{Cap}_n(C)=\inf\Bigl\{\|u\|_{L^2(\mu_n)}^2+\mathcal{E}_n(u,u):u\in
  H^1_{\mathbb{R}}(\mu_n), u\geq1\text{ on a neighborhood of $C$}\Bigr\}$.
\end{proof}

\section{Hermite orthogonal polynomials}
\label{se:hermite}

Set $\mathbb{N}=\{0,1,2,\ldots\}$. Let ${(H_{p,q})}_{p,q\in\mathbb{N}^n}$ be
the multivariate complex Hermite polynomials,
\begin{equation}
  H_{p,q}(z,\bar z)
  =\frac{(-1)^{|p|+|q|}}{\sqrt{p!q!}n^{\frac{|p|+|q|}{2}}}
  \mathrm{e}^{n|z|^2}\bar\partial_{z_1}^{p_1}\cdots\bar\partial_{z_n}^{p_n}\partial_{z_1}^{q_1}\cdots\partial_{z_n}^{q_n}\mathrm{e}^{-n|z|^2}
\end{equation}
where $|p|=p_1+\cdots+p_n$, $|q|=q_1+\cdots+q_n$, $p!=p_1!\cdots p_n!$, and
$q!=q_1!\cdots q_n!$. This is known as the Rodrigues formula for the
multivariate Hermite polynomials with respect to $\gamma_n$. This family of
polynomials forms an orthonormal basis\footnote{Also obtained with the
  Gram--Schmidt algorithm from an ordering of the algebraic basis
  ${(z^a\bar z^b)}_{a,b\in\mathbb{N}^n}$ of $\mathbb{C}[z,\bar z]$.} of
$L^2(\gamma_n)$,
\begin{equation}
  \int_{\mathbb{C}^n}H_{p,q}\overline{H_{p',q'}}\mathrm{d}\gamma_n
  =\mathbf{1}_{(p,q)=(p',q')}
  \quad\text{and}\quad
  \mathcal{H}_n=\overline{\operatorname{span}}\{H_{p,0}:p\in\mathbb N^n\}.
\end{equation}
Indeed, repeated Gaussian integration by parts shows that the family
${(H_{p,q})}_{p,q\in\mathbb{N}^n}$ is orthonormal. Moreover, $H_{p,q}$ is
triangular with leading monomial
$\frac{n^{(|p|+|q|)/2}}{\sqrt{p!q!}}z^p\overline z^q$. Since polynomials are
dense in $L^2(\gamma_n)$, triangularity gives completeness. Thus
${(H_{p,q})}_{p,q\in\mathbb{N}^n}$ is an orthonormal basis. Direct
differentiation of the Rodrigues formula gives the ``lowering relations''
\begin{equation}\label{eq:hermite:lowering}
  \bar\partial_{z_j}H_{p,q}=\sqrt{nq_j}\,H_{p,q-e_j}
  \quad\text{and}\quad
  \partial_{z_j}H_{p,q}=\sqrt{np_j}\,H_{p-e_j,q}
\end{equation}
with the convention that a polynomial with a negative index is zero.

The notation $e_j$ stands for the $j$-th element of the canonical basis of $\mathbb{R}^n$.

We denote by ${(h_{a,b})}_{a,b\in\mathbb{N}}$ the univariate complex Hermite
polynomials associated to the probability measure
$\mathcal{N}_{\mathbb{C}}(0,1/n)=\frac{n}{\pi}\mathrm{e}^{-n|z|^2}\mathrm{d}z$
on $\mathbb{C}$, the one-dimensional complex marginal of the probability
distribution
$\gamma_n=\mathcal{N}_{\mathbb{C}^n}(0,1/n)=\mathcal{N}_{\mathbb{C}}(0,1/n)^{\otimes
  n}$. We have $h_{a,b}(z_1)=H_{ae_1,be_1}(z_1,\ldots,z_n)$ for all $n\geq1$
and all $z_1,\ldots,z_n\in\mathbb{C}$. Their Rodrigues formula is
\begin{equation}\label{eq:one-variable-complex-hermite}
  h_{a,b}(z,\overline z)
  =
  \frac{(-1)^{a+b}}
  {\sqrt{a!b!}\,n^{(a+b)/2}}
  \mathrm{e}^{n|z|^2}
  \bar\partial_z^a\partial_z^b
  \mathrm{e}^{-n|z|^2}.
\end{equation}
Repeated Gaussian integration by parts shows that they are orthogonal and form
an orthonormal basis of $L^2(\mathcal{N}_{\mathbb{C}}(0,1/n))$. Moreover, we
have the lowering relations
\begin{equation}\label{eq:one-variable-hermite-lowering}
  \bar\partial_z h_{a,b} = \sqrt{nb}\,h_{a,b-1}
  \quad\text{and}\quad
  \partial_z h_{a,b} = \sqrt{na}\,h_{a-1,b},
\end{equation}
with the convention that $h_{a,-1}=h_{-1,b}=0$. Furthermore, we have
\begin{equation}\label{eq:holomorphic-hermite}
  h_{a,0}(z)=\sqrt{\frac{n^a}{a!}}z^a
  \quad\text{and}\quad
  h_{0,b}(z)=\sqrt{\frac{n^b}{b!}}\bar{z}^b.
\end{equation}

\section{Numerical experiments}

\begin{figure}[htbp]
  \centering
  \includegraphics[width=.49\textwidth]{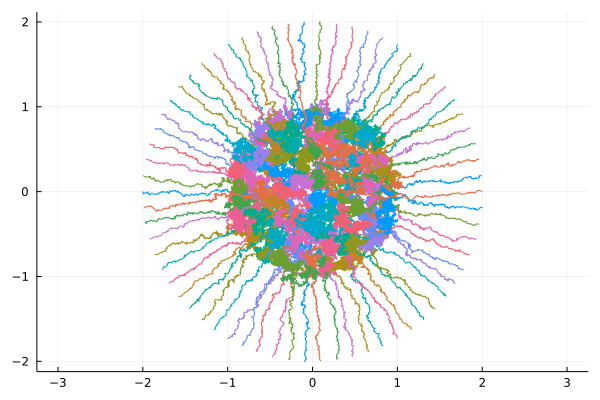}
  \includegraphics[width=.49\textwidth]{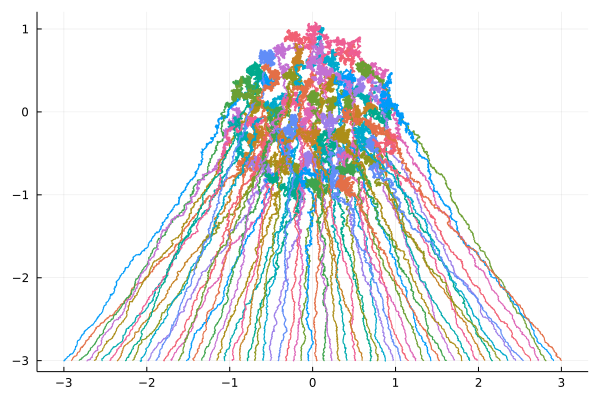}\\
  \caption{The overdamped Langevin dynamics for $n=66$ and two kinds of
    initial conditions (left side: uniformly spaced on a circle, right side:
    on a segment). The attraction for the unit disc is the circular law
    phenomenon, the planar analogue of the semicircle law, already noticed at
    the static level by \cite{Ginibre1965}, see also \cite{ChafaiAspects} and
    \cite{ByunForresterGinibre} for a modern overview.}
\end{figure}

\bibliographystyle{abbrvnat}
{\footnotesize\bibliography{poing}}

\end{document}